\documentclass[12pt]{amsart}  

\usepackage{graphicx}
\usepackage{here} 
\usepackage{array} 
\usepackage{amstext,amscd}
\usepackage{amsthm, cases}

\usepackage{vmargin}
\usepackage{amssymb}
\usepackage{mathrsfs}
\usepackage[all]{xy}
\usepackage[usenames,dvipsnames]{color}
\usepackage{soul}
\RequirePackage[colorlinks,linkcolor=blue,citecolor=LimeGreen,urlcolor=red]{hyperref} 
\usepackage{amsmath}

\newtheorem{theorem}{Theorem}[section]

\newtheorem{cor}[theorem]{Corollary}

\newtheorem{lemma}[theorem]{Lemma}

\newtheorem{definition}[theorem]{Definition}

\newtheorem{prop}[theorem]{Proposition}
\newtheorem{remark}[theorem]{Remark}

\numberwithin{equation}{section}

\newcommand{\R}{\mathbb{R}}

\newcommand{\T}{\mathbb{T}}
\renewcommand{\S}{{\mathbb S}}

\newcommand{\function}[5]{\begin{array}[t]{lrcl}
#1 : & #2 & \longrightarrow & #3 \\
     & #4 & \longmapsto & #5
\end{array}}
\newcommand{\func}[3]{#1 : #2 \longrightarrow #3}

\newcommand{\disp}{\displaystyle}
\newcommand{\abs}[1]{\left|#1\right|}
\newcommand{\eps}{\varepsilon}
\newcommand{\norm}[1]{\left\|#1\right\|}

\renewcommand{\leq}{\leqslant}
\renewcommand{\geq}{\geqslant}
\renewcommand{\bar}{\overline}
\renewcommand{\tilde}{\widetilde}
\newcommand{\pa}[1]{\left(#1\right)}
\newcommand{\cro}[1]{\left[#1\right]}
\newcommand{\br}[1]{\left\{#1\right\}}
\newcommand\restr[2]{{
  \left.\kern-\nulldelimiterspace 
  #1 
  \right|_{ #2} 
  }}
\DeclareMathOperator*{\esssup}{ess\,sup}

\makeatletter
\def\namedlabel#1#2{\begingroup
    #2%
    \def\@currentlabel{#2}%
    \phantomsection\label{#1}\endgroup
}
\makeatother

\newcommand{\beqar}{\begin{eqnarray*}}
\newcommand{\eeqar}{\end{eqnarray*}}

\newcommand{\beqarl}{\begin{eqnarray}}
\newcommand{\eeqarl}{\end{eqnarray}}

\newcommand{\be}{\begin{equation}}
\newcommand{\ee}{\end{equation}}

\newcommand{\xin}{X^{i,N}}
\newcommand{\oin}{V^{i,N}}
\newcommand{\bxin}{\bar{X}^{i,N}}
\newcommand{\boin}{\bar{V}^{i,N}}

\newcommand{\lp}{\left(}
\newcommand{\rp}{\right)}

\def\signmb{\bigskip \begin{center} {\sc
Marc Briant\par\vspace{3mm}
Universit\'e de Paris,\par 
MAP5, CNRS, \par 
F-75006 Paris, France \par
\vspace{3mm}
e-mail:} \tt{briant.maths@gmail.com} \end{center}}

\def\signsm{\bigskip \begin{center} {\sc
Sara Merino Aceituno\par\vspace{3mm}
Faculty of Mathematics, University of Vienna,\par
Oskar-Morgenstern-Platz 1, 1020 Vienna, Austria\par
Department of Mathematics, University of Sussex\par
Brighton, BN1 9RH, United Kingdom\par
\vspace{3mm}
e-mail:} \tt{sara.merino@univie.ac.at} \end{center}}

\def\signad{\bigskip \begin{center} {\sc
Antoine Diez\par\vspace{3mm}
Department of Mathematics, Imperial College London,\par
\par
South Kensington Campus\par
London SW7 2AZ, UK\par
\vspace{3mm}
e-mail:} \tt{antoine.diez18@imperial.ac.uk} \end{center}}

\begin{document} 

\title[General Vicsek models: particle and kinetic equations]{Cauchy theory  for general kinetic Vicsek models in collective dynamics and  mean-field limit approximations}
\author{Marc Briant, Antoine Diez, Sara Merino-Aceituno}

\begin{abstract}
In this paper we provide a local Cauchy theory both on the torus and in the whole space for general Vicsek dynamics at the kinetic level. We consider rather general interaction kernels, nonlinear viscosity and nonlinear friction. Particularly, we include normalised kernels which display a singularity when the flux of particles vanishes. Thus, in terms of the Cauchy theory for the kinetic equation, we extend to more general interactions and complete the program initiated in \cite{GamKan} (where the authors assume that the singularity does not take place) and in \cite{FigKanMor} (where the authors prove that the singularity does not happen in the space homogeneous case). 
Moreover, we derive an explicit lower time of existence as well as a global existence criterion that is applicable, among other cases, to obtain a long time theory for non-renormalised kernels and for the original Vicsek problem without any \textit{a priori} assumptions. On the second part of the paper, we also establish the mean-field limit in the large particle limit for an approximated (regularized) system that coincides with the original one whenever the flux does not vanish. Based on the results proved for the limit kinetic equation, we prove that for short times, the probability that the dynamics of this approximated particle system coincides with the original singular dynamics tends to one in the many particle limit.

\end{abstract}

\maketitle

\vspace*{10mm}



\bigskip

\textbf{Acknowledgements.}
The work of A. Diez is supported by an EPSRC-Roth scholarship cofounded by the  Engineering and Physical Sciences Research Council and the Department of Mathematics at Imperial College London.
S. Merino-Aceituno is supported by the Vienna Science and Technology Fund (WWTF) with a Vienna Research
Groups for Young Investigators, grant VRG17-014.\\
The authors also thank Amic Frouvelle for useful discussions.

\tableofcontents

\section{Introduction} \label{sec:intro}

\subsection{Motivation}

The emergence of collective motions among a group of  self-propelled agents is receiving a great deal of attention: flocks of birds \cite{hildenbrant}, schools of fish \cite{hemelrijk2012schools}, pedestrian dynamics \cite{maury2010macroscopic}, micro-swimmers \cite{david2015mass}. Many models have been proposed to explain these collective behaviours. Among them, the Vicsek model \cite{degond2008continuum,vicsek1995novel} is one of the most studied ones. In this model all agents move  at a constant speed while trying to adopt the same velocity as their neighbours, up to some noise.  

In this introductory section we present the discrete dynamics for the Vicsek model as presented in \cite{degond2008continuum} as well as its corresponding kinetic equation for the time-evolution of the distribution of the particles. Up to now, there were no complete rigorous results on the existence of solutions and the derivation of the kinetic equation (mean-field limit) for the Vicsek dynamics presented in \cite{degond2008continuum}. The reason for this is the presence of a singularity in the dynamics that is reached when the local average velocity of the agents vanishes.  

In contrast to \cite{degond2008continuum}, one can define a Vicsek model without a singularity \cite{BolCanCar2,DegFroLiu3,DegFrouLiu2,FrouLiu}, for different modelling choices. However, these
 different modelling choices for the Vicsek model have profound mathematical implications (like appearance of phase transitions). For this reason, in this article we investigate a general form of the Vicsek model that includes the forms presented in \cite{degond2008continuum} (with a singularity), in \cite{BolCanCar2,FrouLiu} (without singularity) 
as well as further extensions of the model like the ones
 in \cite{BriMeuNav,
DegFrouLiu2,degond2015multi,frouvelle2012continuum}. We give conditions to have existence of solutions for the equations  and investigate approximations to the mean-field limit for this general class of Vicsek models. All of these are detailed in the following sections.

\subsection{Particular case: Vicsek model as in \cite{degond2008continuum} and \cite{BolCanCar2,FrouLiu}}
  In what follows we assume that the agents move in a domain $\Omega$ with $\Omega$ being either the $d$-dimensional torus $\T^d$ or the whole space $\R^d$ for $d\geq 1$. The orientations of the agents are given elements on the sphere $\S^{d-1}$. Throughout the article, for a given vector $\omega\in\R^d$ we denote $\mathbf{P}_{\omega^\bot}$ the orthogonal projection in $\R^d$ onto $\pa{\R\omega}^\bot$. The time-continuous Vicsek model \cite{degond2008continuum} considers $N$ agents characterised by their positions $X^{i,N}(t)\in \Omega$, $i=1,\hdots, N$ and orientations $\omega^{i,N}(t)\in \mathbb{S}^{d-1}$ over time $t\geq 0$. The evolution of the system is given by the following Stratonovich stochastic differential equation, where $c, \nu,\sigma>0$ are positive constants:
\beqarl
dX^{i,N} &=& c \omega^{i,N} \, dt, \label{eq:Vicsek_X}\\
d\omega^{i,N}&=& \nu \nabla_{\omega^{i,N}} (\omega^{i,N} \cdot \bar \omega_i) \, dt + \mathbf{P}_{(\omega^{i,N})^\perp}\circ \sqrt{2\sigma}dB_t^i. \label{eq:Vicsek_omega}
\eeqarl
We explain next this system of equations. Eq. \eqref{eq:Vicsek_X} describes the transport of the agents: agent $i$ moves in the orientation $\omega^{i,N}$ at speed $c>0$. Eq. \eqref{eq:Vicsek_omega} gives the evolution of the orientations over time. It includes two competing forces. On one hand the first term in the form of a gradient represents organised motion: agents try to adopt the same orientation.  The gradient $\nabla_\omega$ is the gradient on the sphere and the term $\bar \omega_i$ represents the average orientation of the neighbours around agent $i$. Therefore, the first term in the right-hand-side of Eq. \eqref{eq:Vicsek_omega} is a gradient flow which relaxes the value of the orientation $\omega^{i,N}$ towards the mean orientation of the neighbouring particles $\bar \omega_i$. The constant $\nu>0$ gives the speed of this relaxation. We will comment later on the different choices to compute the average orientation $\bar \omega_i$. These different choices give rise to the different models in \cite{degond2008continuum} and \cite{BolCanCar2,FrouLiu}.

On the other hand $B^i_t$, $i=1,\hdots, N$ are $N$ independent Brownian motions in $\R^d$ and they introduce noise in the dynamics, driving particles away from organised motion. The constant $\sigma>0$ gives the intensity of the noise. The symbol $'\circ'$ is used to specify that  Eq. \eqref{eq:Vicsek_omega} has to be understood in the Stratonovich sense. This ensures that $\omega_i(t)\in \mathbb{S}^{d-1}$ for all times where the solution is defined (this will be proven later).

Now, formally at least, one can compute the time-evolution for the distribution of agents $f=f(t,x,\omega)$ in space $x\in \Omega$ and orientations $\omega\in \mathbb{S}^{d-1}$ at time $t\geq 0$ as the number of particles $N\to \infty$  \cite{degond2008continuum}. The dynamics for the distribution $f$ is given by the following kinetic equation:
\be 
\label{eq:Vicsek_kinetic}
\partial_t f + c\omega\cdot \nabla_\omega f = \nabla_\omega \cdot \left[- \nu \mathbf{P}_{\omega^\perp}\bar\omega_f\, f + \sigma \nabla_\omega f\right] =: L(f),
\ee
where $\nabla_\omega \cdot$ denotes the divergence on the sphere $\mathbb{S}^{d-1}$ and where $\bar\omega_f(t,x)$ represents the average orientation of the particles around position $x$ at time $t$. Notice that the projection term appears due to the fact that
$$\nabla_\omega (\omega\cdot \bar \omega ) = \mathbf{P}_{\omega^\perp}\bar \omega, \qquad \mbox{ for }\bar \omega \in \R^d \mbox{ fixed}.$$
One can show that the operator $L$ on the right-hand-side of \eqref{eq:Vicsek_kinetic} can be rewritten as a non-linear Fokker-Planck operator:
\be \label{eq:Fokker-Planck}
L(f):= \sigma\nabla_\omega \cdot \left[ M_{\bar\omega_f}\nabla_\omega \lp \frac{f}{M_{\bar\omega_f}}\rp\right],
\ee
with
\be \label{eq:von-Mises}
M_v(\omega) = \frac{\exp(\frac{\nu}{\sigma}(\omega\cdot v))}{Z}, \qquad Z = \int_{\mathbb{S}^{d-1}}\exp\lp \frac{\nu}{\sigma}(\omega\cdot v)\rp\, d\omega,
\ee
for any $v\in \mathbb{S}^{d-1}$. The function $M_v$  is a probability density on the sphere named von-Mises distribution. Notice that its mean is in the orientation given by $v$. When $\sigma \to \infty$ (large-noise limit) the von-Mises distribution converges to the uniform distribution on the sphere.

\bigskip
We comment now on the different choices to define the average orientation around an agent $i$, denoted by $\bar\omega_i$ in Eq. \eqref{eq:Vicsek_omega}. The choice considered has profound implications in the dynamics of the particles and in the derivation of the kinetic equation \eqref{eq:Vicsek_kinetic} (mean-field limit) and the derivation of macroscopic equations (equations for the particle density and mean orientation). In the mathematics literature, two options were first considered given, in a compact way, by:
\be
\bar \omega_i = \frac{J^N_i}{\alpha + (1-\alpha) |J^N_i|}, \qquad J^N_i(t) = \sum_{j\neq i}^NK(|X_i(t)-X_j(t)|)\, \omega_j(t). \label{eq:Vicsek_average}
\ee
The two options correspond to the cases $\alpha=0$ (as presented in \cite{degond2008continuum}) or $\alpha=1$ (as presented in \cite{BolCanCar2,FrouLiu}). The kernel $K\geq 0$ is an interaction kernel that represents the weights given to the neighbouring particles depending on their distance to particle $i$. A classical choice is to take $K$ the indicator function of a ball of radius $R>0$. The flux $J^N_i$ is, indeed, an average of the orientations of the neighbouring particles. In the kinetic equation these choices define the operator $\bar\omega_f$ as:
\be \label{eq:average_kinetic}
\bar \omega_f(t,x) =\frac{\mathbf{J}_f(t,x)}{\alpha + (1-\alpha)|\mathbf{J}_f(t,x)|}, \qquad \mathbf{J}_f(t,x) = \int_{\Omega\times\mathbb{S}^{d-1}} K(|x-y|)\, \omega f(t,y,\omega)\, dyd\omega.
\ee

\medskip

\noindent We compare next, the mathematical implications of these two choices:

\noindent\textbf{Case $\alpha=1$ (non-normalisation).} If $\alpha=1$, the average orientation $\bar\omega_i ~=~ J_i^N~\notin~\mathbb{S}^{d-1}$ is not a unit vector. This also holds in the kinetic equation \eqref{eq:Vicsek_kinetic} for $\bar\omega_f = \mathbf{J}_f\notin \S^{d-1}$ in \eqref{eq:average_kinetic}.  However, the case $\alpha =1$ removes the singularity when $|\mathbf{J}_f|=0$ or $|J^N_i|=0$. In \cite{FrouLiu} the authors prove the well-posedness of the space-homogeneous kinetic equation in any Sobolev space and in \cite{BolCanCar1}  the authors prove the mean-field limit (here we will recover these results). As a counterpart, though, this choice for the average $\bar \omega_i$ makes the derivation of macroscopic equations for the density of the particles $\rho=\rho(t,x)$ and the mean orientation $\bar \omega=\bar\omega(t,x)$ of the agents  more complex than in the case $\alpha=0$.
Specifically, two equilibria exist for the mean particle orientation $\bar \omega$ depending if $\mathbf{J}_f \neq 0$ or $\mathbf{J}_f=0$ \cite{DegFroLiu3,DegFrouLiu2,FrouLiu}.  In loose terms, if $\mathbf{J}_f\neq 0$,
 then the von-Mises equilibria $M_{\mathbf{J}_f}$ is an equilibrium. The other equilibrium is given by the uniform distribution on the sphere, which corresponds to the case $\mathbf{J}_f=0$ in $M_{\mathbf{J}_f}$.  In different spatial regions (depending on the particle density $\rho$) one of these two equilibria is stable.  Consequently, this gives rise to a bifurcation or phase transition: in some spatial regions  the mean orientation of the agents is $\bar \omega =0$ (corresponding to disordered dynamics) and in other spatial regions the mean orientation is given by a unit vector $|\bar\omega|=1$ (corresponding to ordered dynamics or flocking). The interested reader can find all the details in \cite{BriMeuNav,DegFroLiu3,DegFrouLiu2,FrouLiu}. The presence of phase transitions enriches the dynamics, in the sense that allows for a wider variety of patterns to arise. Understanding phase transitions is also of great interest in the Physics community \cite{vicsek1995novel}, who was the first one to introduce the Vicsek model. In particular, pattern formation has been studied through the simulation of particle simulations for the Vicsek model and its variations. Band formation arises under some parameter conditions \cite{caussin2014emergent,chate2008modeling,chate2008collective}. Phase transitions could be key to explaining the emergence of this band formation.
\medskip

\noindent \textbf{Case $\alpha=0$ (normalisation).} The first mathematical works on the Vicsek model correspond to this  case \cite{degond2008continuum}. The choice $\alpha=0$ comes with the difficulty of dealing with singularities when the flux $J^N_i =0$ (another example of collective dynamics with singularities, different from the Vicsek model can be found in \cite{poyato2019filippov} for the Kuramoto model). However, assuming that the flux $\mathbf{J}_f$ in the kinetic equation does not vanish along the dynamics, then the macroscopic equations can be obtained more easily than in the case $\alpha=1$ because there is a unique equilibria: the von-Mises distribution in \eqref{eq:von-Mises}. Then the macroscopic equations for the particle density $\rho=\rho(t,x)$ and mean orientation $\bar \omega=\bar \omega(t,x)\in \mathbb{S}^{d-1}$ correspond to the Self-Organised Hydrodynamics (SOH) given in \cite{degond2008continuum}. This was the first formal derivation of the SOH dynamics. See also \cite{jiang2016hydrodynamic,zhang2017local} as well as \cite{jiang2020coupled} for later rigorous results. In this scenario a rigorous mean-field limit to derive the kinetic equation was missing as well as a Cauchy theory for the particle dynamics and the kinetic equation. Here we will not prove the mean-field limit starting from a particle system of the form \eqref{eq:Vicsek_X} -\eqref{eq:Vicsek_omega}, but from a modified system that we term `approximated particle system'. This system does not have a singularity when the flux vanishes and therefore, it can be thought of as a regularization of the original particle dynamics. Proving the mean-field limit for particle systems with non-regular coefficients or with singularities is in general a difficult problem and, to the authors knowledge, only very few results exist and focus on specific problems (see e.g. \cite{jabin2017mean}).  Here, we will  prove that with a probability which tends to one in the many-particle limit, for short times, solutions of the approximated particle system are also solutions of the Vicsek particle system.
We investigate these questions in the present article:

\begin{theorem}[Cauchy theory for general Vicsek models with normalisation and mean-field from approximated particle dynamics]
Suppose that $\alpha =0$ (normalised case). Suppose that the kernel $K$ is Lipschitz, bounded and that $f(0,x,\omega)$  satisfies assumptions (i) and (ii) in Theorem \ref{theo:local cauchy}. Then, there exists a unique local-in-time solution to the kinetic equation \eqref{eq:average_kinetic}. Moreover, the kinetic equation equation \eqref{eq:average_kinetic} can be obtained as the mean-field limit of an `approximated' particle dynamics whose solutions are, with a probability which tends to one in the many-particle limit, also solutions to the particle Vicsek dynamics \eqref{eq:Vicsek_X}-\eqref{eq:Vicsek_omega}.
\end{theorem}

The precise definition of `solution', of the approximated particle system and what we mean with `a probability which tends to one in the many-particle limit' will be explained in Th. \ref{theo:local cauchy} and in Sec. \ref{subsec:particles intro}, respectively. 
As mentioned before, the difficulty in proving the previous theorem is to show that there exists at least some time interval $t\in[0,T]$ such that $\mathbf{J}_f\neq 0$ and that for $N$ large enough the probability that $|J^N_i|>0$ goes to 1 as $N\to\infty$ for all $i=1,\hdots,N$ and $t\in[0,T]$. To prove a lower bound in the flux $\mathbf{J}_f\neq 0$ is, therefore, key. Indeed, in  \cite{GamKan} the authors manage to prove well-posedness for the kinetic equation \eqref{eq:Vicsek_kinetic} assuming precisely that \textit{a priori} all solutions satisfy $|\mathbf{J}_f|(t,x,\omega)\geq a >0$ for all times, positions and orientations. In \cite{FigKanMor} the authors prove well-posedness in the space homogeneous setting as long as initially $|\mathbf{J}_{f_0}|>0$ (plus some other assumptions). This has been followed by  further extensions  in \cite{kang2016dynamics}. The space-homogeneous case corresponds to spatially local kernels, i.e., $K(x-y)=\delta_0(x-y)$ (delta distribution).

\medskip
To conclude this part, the choice on how to define the average orientation $\bar\omega_i$ has profound mathematical and modelling implications (like dealing with singularities or with phase transitions). For this reason, in this article we will consider a general class of Vicsek models to encompass different modelling choices - existing or yet to be developed. 

\subsection{General forms of the Vicsek model considered}
As we saw in the previous section, modelling choices in the Vicsek model may produce important differences in the mathematical properties of the equations. For this reason, in this article we consider a  wide class of Vicsek models based not only on the choice of the averaged orientation \eqref{eq:Vicsek_omega} but also on the particular shape of the friction $\nu$, the viscosity $\sigma$ and the interaction kernel $K$.  Specifically,  we will allow $\nu$ and $\sigma$ to be functions of $f$ (the solution to the kinetic equation) at the kinetic level and functions of the empirical distribution (see Eq. \eqref{eq:empirical_distribution})
at the particle level. The interaction kernel $K$ can take the general form $K=K(t,x,x_*, \omega, \omega_*)$. The precise assumptions on $\nu$, $\sigma$ and $K$ can be found in Section \ref{sec:functional_framework}.

As a consequence, the results presented here apply to a wide breath of Vicsek-types models, like the ones in \cite{DegFroLiu3,
DegFrouLiu2,degond2015multi,frouvelle2012continuum}.  In \cite{DegFroLiu3,DegFrouLiu2} the authors consider functions $\nu=\nu(|\mathbf{J}_f|)$, $\sigma = \sigma(|\mathbf{J}_f|)$ that depend on the flux of the particles; in \cite{frouvelle2012continuum} the author considers a kernel $K$ that is not isotropic; in \cite{degond2015multi} the authors consider a Vicsek-type model for alignment of discs. Beyond this, the approach presented here can be  adapted to investigate other models like the ones in \cite{BriMeuNav,degond2019phase,degond2014hydrodynamics,
DegFrouMerTre,
degond2017alignment,
degond2017hui,degond2018age,
degond2019nematic,
degond2019coupled,
degond2011macroscopic,navoret2013two}. In \cite{degond2014hydrodynamics}  the authors couple the Vicsek model with a Kuramoto model. In \cite{navoret2013two} the author considers a Vicsek model with two species having different velocities. The models in \cite{BriMeuNav,DegFrouMerTre,degond2017hui,degond2018age,degond2019nematic} describe  collective motion based on nematic alignment (where particles align by adopting the same direction of motion and not necessarily the same orientation).  BGK versions of Vicsek-type models are considered in  \cite{degond2019phase,degond2017alignment}. In  \cite{degond2019coupled}  the authors couple the Vicsek model with Stokes equations to model micro-swimmers. Finally, in \cite{degond2011macroscopic} a model for the persistent turning walker with curvature `alignment' is presented.

\subsection{Aims of the paper}
In this paper, for a general form of the Vicsek model, we aim at:
\begin{itemize}
\item[(i)] establishing the well-posedness, at least locally in time, for the kinetic model (Theorem \ref{theo:local cauchy}); existence of solutions will be proven in Lebesgue spaces;
\item[(ii)] rigorously proving that the kinetic equation can be derived as the mean-field limit of some `approximated' agent-based dynamics  in the limit $N \to +\infty$ (Theorems \ref{th:existence_solutions}, \ref{th:convergence_process}, \ref{cor:mean-field}). Also show  that  with a probability which tends to one in the many-particle limit realizations of the `approximated' particle system are solutions of the original Vicsek system (Theorem \ref{th:NormalizedCase});
\item[(iii)] deriving a criterion that characterises the global-in-time well-posedness of these systems (points (a) and (b) in Theorem \ref{theo:local cauchy}) and apply it to several interesting cases (Section \ref{subsec:global examples}). Along the way we shall also open the discussion to long-time behaviour and construct a free energy but we shall not investigate further (Section \ref{subsec:convergence to equilibrium}).
\end{itemize}
We aim at giving constructive proofs rather than weak-compactness arguments and at working in  Lebesgue spaces rather than more regular Sobolev spaces. Our strategy is thus to prove a fixed point contraction theory for the kinetic equation \eqref{eq:general kinetic orthogonal} and, to prove the mean-field limit, using  a coupling approach popularised by Sznitmann \cite{BolCanCar2,carmona2016lectures,chaintron2021propagation,sznitman1991topics}. 

\subsection{Structure of the paper}

The paper is structured as follows.
Section \ref{sec:results} describes our main results, first in the kinetic framework in Section \ref{sec:results_kinetic} and second for the particle dynamics and its mean-field limit in Section \ref{subsec:particles intro}. The proofs for the results in Section \ref{sec:results_kinetic} are given in Section \ref{sec:kinetic} and the proofs of the results in Section \ref{subsec:particles intro} are given in Section \ref{sec:particles}.  In the Appendix \ref{sec:carmona} the reader can find known results on stochastic differential equations and mean-field limits that are used in Section \ref{sec:particles} and that have been added here for the sake of completeness. The results in Appendix \ref{sec:carmona} are mainly from \cite{carmona2016lectures}.

\section{Main results and strategy}
\label{sec:results}
\subsection{Functional framework and notations}
\label{sec:functional_framework}

First let us give some notations for functional spaces. We shall work on Lebesgue spaces indexed by the variable into consideration: for any $p$ in $[1,+\infty)$ we denote
$$\norm{f}_{L^p_\omega} = \pa{\int_{\S^{d-1}}\abs{f(\omega)}^pd\omega}^{\frac{1}{p}} \quad\mbox{and}\quad \norm{f}_{L^p_x} = \pa{\int_{\Omega}\abs{f(x)}^pdx}^{\frac{1}{p}}$$
and for a time-dependent function for any $T>0$
$$\norm{f}_{L^p_t} = \pa{\int_{\R^+}\abs{f(t)}^pdt}^{\frac{1}{p}}\quad\mbox{and}\quad \norm{f}_{L^p_{[0,T)}} = \pa{\int_{0}^T\abs{f(t)}^pdt}^{\frac{1}{p}}.$$
Finally we denote the several variable Lebesgue spaces for any $q$ and $r$ in $[1,+\infty)$
$$\norm{f}_{L^r_{[0,T)}L^q_xL^p_\omega} = \pa{\int_0^T\pa{\int_{\Omega}\norm{f(t,x,\cdot)}_{L^p_\omega}^qdx}^{\frac{r}{q}}dt}^{\frac{1}{r}}$$
with direct modifications for $p$, $q$ or $r$ being $+\infty$. Note that when two indexes are the same we shall use the short-hand notation $L^p_xL^p_\omega = L^p_{x,\omega}$. \newline

For the mean-field limit results, we will consider the space $\mathcal{H}=(\Omega \times \R^d \times \mathcal{P}_2(\Omega\times \R^d)$), where $\mathcal{P}_2(\Omega\times \R^d)$ is the space of probability measures in $\Omega\times \R^d$ with finite second-order moment, i.e., $\mu \in \mathcal{P}_2(\Omega \times \R^d)$ if it fulfils
$$\int_{\Omega \times \R^d} |z|^2 \, \mu(dz) \, <\infty.$$
For $\mu, \mu'\in \mathcal{P}_2 (\Omega \times \R^d)$, the 2-Wasserstein distance is given by
\beqarl
W_2(\mu, \mu') &=& \inf\Big\{\left[\int_{(\Omega\times\R^{d})^2}|z-z'|^2 \, \pi(dz, dz')\right]^{1/2} \,;\\
&& \qquad \qquad \quad \pi\in \mathcal{P}(\Omega\times \R^{d})^2 \mbox{ with marginals } \mu \mbox{ and }\mu'\Big\}. \label{eq:wasserstein_distance}
\eeqarl
The distance $W_2$ induces the topology of weak convergence of measures and the convergence of all the moments of order up to 2 \cite{carmona2016lectures} (see also \cite[p. 83]{villani2008optimal}).
The space $\mathcal{H}$ is a metric space with the distance $d_\mathcal{H}$ given by 
$$d_{\mathcal{H}}\Big((x,v,m), (y,w,p) \Big) = |x-y|+|v-w|+ W_2(m, p).$$

\bigskip

\noindent The hypotheses we shall make on the nonlinearities are the following.

\renewcommand{\labelenumi}{(H\theenumi)}
\begin{enumerate}
\item The interaction kernel $ \mathbf{K}$  is  Lipschitz, regular and bounded in all the variables. More precisely, it is assumed that $\mathbf{K}(t,x,x_*,\omega,\omega_*) \in W^{1,\infty}_{t,x_*}W^{2,\infty}_{\omega_*}L^\infty_{x,\omega}$, and in the case $\Omega = \mathbb{R}^d$ we further assume that $\mathbf{K} \in L^\infty_{t,x,\omega}L^2_{x_*,\omega_*}$; \label{HK}
\item The viscosity satisfies that $\func{\sigma}{L^2_{[0,T],x,\omega}}{L^\infty_{[0,T],x,\omega}}$ is bounded from below and Lipschitz uniformly for any $T>0$ in the sense there exists $\sigma_0$, $\sigma_\infty$ and $\sigma_{\tiny{lip}}$ such that for any $T>0$:
\begin{equation*}
\begin{split}
\forall (f,g) \in \pa{L^2_{[0,T],x,\omega}}^2,\quad & 0<\sigma_0 \leq \sigma(f)\leq \sigma_\infty
\\& \norm{\sigma(f)-\sigma(g)}_{L^\infty_{[0,T],x,\omega}} \leq \sigma_{\mbox{\tiny{lip}}}\norm{f-g}_{L^2_{[0,T],x,\omega}}.
\end{split}
\end{equation*}\label{Hsigma}
\item The friction is local in time and Lipschitz that is $\nu(f)(t,x,v) = \nu(f(t,\cdot,\cdot))(x,v)$ where $\func{\nu}{L^2_{x,\omega}}{L^\infty_{x}W^{1,\infty}_\omega}$:
\begin{equation*}
\begin{split}
\forall (f,g) \in \pa{L^2_{x,\omega}}^2,\quad & \norm{\nu(f)}_{L^\infty_{x}W^{1,\infty}_\omega}\leq \nu_\infty
\\& \norm{\nu(f)-\nu(g)}_{L^\infty_{x,\omega}} \leq \nu_{\mbox{\tiny{lip}}}\norm{f-g}_{L^2_{x,\omega}}.
\end{split}
\end{equation*} \label{Hnu}
\item For the mean-field limit results we will assume further that $\sigma$, $\nu$ and $\nabla_\omega \sigma$ are Lipschitz and bounded in  $W_2$. 
 \label{Hmeanfield}
\end{enumerate}

\bigskip
\begin{remark}
\begin{itemize}
\item 
Note that $(H\ref{Hsigma})$ and $(H\ref{Hnu})$ are satisfied when $\sigma$ and $\nu$ do not depend on $f$, which is the case in most models so far. Ultimately, one would want $\sigma(f)$ to be local in time and Lipschitz, like $\nu(f)$. We think that this could be achieved with standard parabolic regularity methods for more regular initial data and $\sigma(f)$. The main issue for closing a fixed point argument is the speed of convergence to $0$ of the gain of regularity of the solutions $f$ that we did not manage to quantify, see Proposition \ref{prop:dependence coef} and Remark \ref{rem:dependence coef}.
\item 
In applications, typically  the functions $\nu=\nu(t,x,\omega,f)$ and $\sigma=\sigma(t,x,\omega,f)$ are of the form
\begin{equation}
\label{eq:typical_form_nu}
\nu(t,x,\omega,f) = \int_{\Omega \times \mathbb{S}^{d-1}}\tilde{\nu}(t,x,\omega, x_*,\omega_*)\, f(t, x_*, \omega_*)\, dx_* d\omega_*,
\end{equation}
for some function $\tilde \nu$, and analogously for $\sigma$. For example, in \cite{degond2008continuum} the authors consider
$$\nu(t,x,\omega,f) = \omega \cdot \bar \omega_f(t,x), $$
with
\begin{equation} \label{eq:aux_example_flux}
 \bar \omega_f(t,x) = \frac{\tilde{\mathbf{J}}_f}{|\tilde{\mathbf{J}}_f|}(t,x), \quad \tilde{\mathbf{J}}_f(t,x)= \int_{\Omega\times \mathbb{S}^{d-1}}K(x-x_*,\omega_*) f(t,x_*,\omega_*) dx_* d\omega_*.
 \end{equation}
 And in \cite{DegFrouLiu2} the authors consider $\nu$ and $\sigma$ of the form
 $$\nu= \nu(t,x,f) =\nu(|\tilde{\mathbf{J}}_f|(t,x)),$$
 with $\tilde{\mathbf{J}}_f$ given by \eqref{eq:aux_example_flux}.

\end{itemize}
\end{remark}

\subsection{Kinetic point of view: strategy and results}
\label{sec:results_kinetic}

We will consider the following kinetic equation for the distribution function $f=f(t,x,\omega)$ for $(t,x,\omega) \in \R_+\times \Omega\times \R^d$ (with $\Omega$ either the $d$-dimensional torus $\T^d$ or the full space $\R^d$):
\begin{equation}\label{eq:general kinetic orthogonal}
\begin{split}
\partial_t f + c \omega\cdot\nabla_x f &= \nabla_\omega \cdot\pa{\sigma(f) \nabla_\omega f} +\nabla_\omega\cdot\pa{\nu(f)f\mathbf{P}_{\omega^\bot}\cro{\boldsymbol\Psi[f]}},
\end{split}
\end{equation}
where $\nabla_\omega$, $\nabla_\omega\cdot$ are the gradient and divergence operators on the sphere; $c>0$ is a positive constant and $\nu, \sigma, \boldsymbol \Psi$ are given functions.

\begin{remark}
Notice that we allow the function $\nu(f)<0$. For example, in the kinetic equation for the Vicsek model in Eq. \eqref{eq:Vicsek_kinetic} this function corresponds to $-\nu$ for some constant $\nu$ strictly positive.
\end{remark}

\bigskip
\par There exists an interaction kernel $\func{\mathbf{K}}{\R^+\times\Omega^2\times\pa{\S^{d-1}}^2}{\R^d}$ that defines the flux $\boldsymbol J$ as
\begin{equation} \label{eq:flux_definition}
\boldsymbol J[f] (t,x,\omega)=\int_{\Omega\times \S} \mathbf{K}(t,x,x_*,\omega,\omega_*)f(t,x_*,\omega_*)dx_*d\omega_*.
\end{equation}
From the flux $\boldsymbol J$ we define the functional $\boldsymbol\Psi$ as
\begin{equation}\label{eq:general kinetic orthogonal kernel}
\exists \alpha \in [0,1],\quad\boldsymbol\Psi[f](x,\omega) = \frac{\boldsymbol  J[f](t,x,\omega)}{\abs{\boldsymbol J[f](t,x,\omega)}_\alpha},
\end{equation}
where
\begin{equation}\label{norm alpha}
\abs{\boldsymbol J[f](t,x,\omega)}_\alpha = \alpha +(1-\alpha)\abs{\boldsymbol J[f](t,x,\omega)},
\end{equation}
where $\abs{\cdot}$ stands for the norm in $\R^d$. Note that when $\alpha =0$ we talk about a normalised operator since $\abs{\boldsymbol\Psi[f]}= 1$ whereas for $\alpha=1$ we have a complete kernel operator.

The orthogonal kernel non-linearity $\eqref{eq:general kinetic orthogonal}$ we tackle is more general than the gradient-type interaction $\eqref{eq:general kinetic gradient}$. Furthermore, it is also more physically relevant when one wants to derive the Kolmogorov-Vicsek type of kinetic equation from particle systems behaviour \cite{BolCanCar1,BolCanCar2}. The main issue for the kinetic part is the possible degeneracy of $\abs{\boldsymbol\Psi[f]}_\alpha$ as well as the nonlinearity of the dissipativity and viscosity that does not necessarily compensate the aformentionned degeneracy, unlike existing works in the literature. The key part of our strategy is to provide first a quantification of the possible time of degeneracy of the nonlinearity combined to a study of the dependencies over $\sigma$ and $\nu$ at a linear level.

\begin{remark}
In the literature, typically, the operator $K$ is of the form $\mathbf{K}(x,x_*,\omega,\omega_*) = \mathbf{K}(x-x_*,\omega_*)$. In this case it holds that  
$$\mathbf{P}_{\omega^\bot}(\boldsymbol\Psi[f]) = \nabla_\omega\pa{\omega \cdot\int_{\Omega\times\S^{d-1}} \mathbf{K}(x-x_*,\omega_*)f(x_*,\omega_*)dx_*d\omega_*}$$
because for a vector $\mathbf{X}$ in $\R^d$, $\nabla_\omega (\omega\cdot X) = \mathbf{P}_{\omega^\bot}(\mathbf{X})$. Therefore, in this case the Kolmogorov-Vicsek equation $\eqref{eq:general kinetic orthogonal}$ can be  rewritten as a gradient-type interaction
\begin{equation}\label{eq:general kinetic gradient}
\begin{split}
\partial_t f + c \omega\cdot\nabla_x f &= \nabla_\omega \cdot\pa{\sigma(f) \nabla_\omega f} + \nabla_\omega\cdot\pa{\nu(f)f\nabla_\omega\psi[f]}.
\end{split}
\end{equation}
\end{remark}
\bigskip
\begin{theorem}\label{theo:local cauchy}
Let $\Omega$ being either $\T^d$ or $\R^d$, $\alpha$ in $[0,1]$ and $\sigma$, $\nu$ and $\mathbf{K}$ satisfying the hypotheses $(H\ref{HK})-(H\ref{Hsigma})-(H\ref{Hnu})$. Let $p$ belong to $[2,+\infty]$ and $f_0$ be such that
\begin{itemize}
\item[(i)] $f_0$ is a non-negative function in $L^1_{x,\omega} \cap L^p_{x,\omega}$ with mass 
$$\int_{\Omega\times\S^{d-1}}f_0(x,\omega)dxd\omega:=M_0 >0;$$
\item[(ii)]  $\inf\limits_{(x,\omega) \in \Omega\times\S^{d-1}} \abs{\boldsymbol J[f_0](x,\omega)}_\alpha:=J_0  > 0$, where $\abs{\boldsymbol J[f_0]}_\alpha$ was defined in $\eqref{norm alpha}$.
\end{itemize}
There exists a time $T_{\max} >0$, independent of $p$, and a unique weak solution $f$ in $L^2\pa{[0,T_{\max}),L^1_{x,\omega} \cap L^2_{x,\omega}}$ to $\eqref{eq:general kinetic orthogonal}$ with $f_0$ as initial datum. Moreover, $f$ is non-negative on $[0,T_{\max})$, belongs to $L^\infty\pa{[0,T_{\max}),L^p_{x,\omega}} \cap L^2\pa{[0,T_{\max}),L^2_xH^1_\omega}$, preserves the mass $M_0$ and one of the following holds
\begin{itemize}
\item[(a)] $\disp{T_{\max} = +\infty}$
\item[(b)] $\disp{\lim\limits_{t\to T_{\max}^-}\sup\limits_{\Omega \times\S^{d-1}}\frac{\nu(f)}{\abs{\boldsymbol J[f]}_\alpha}=+\infty}$.
\end{itemize}
\end{theorem}
\bigskip

\begin{remark}
Of important note are the following consequences.
\begin{itemize}
\item Our Cauchy theory does not stand on any \textit{a priori} assumption of the solutions and provide an explicit lower bound for $T_{\max}$ (see $\eqref{T0}-\eqref{T1}$). In particular, for $T<T_1$ (where $T_1>0$ is given in \eqref{T1}), it holds that for all $t\in[0,T]$
\be \label{eq:boundJf}
|\boldsymbol J(t,x,\omega,f)| \geq J_0 - K_\infty M_0 T>0,
\ee
where $K_\infty$ is given by \eqref{eq: Kinfty} and $M_0$ is as given in the theorem above.\\
In the sequel we will define $c_*= c_*(T)$ as
\be  \label{eq:defcstart}
c_* = J_0 - K_\infty M_0 T, \quad \mbox{ for }T<T_1.
\ee

\item The theorem above includes all the previous results made in $L^2$ or $L^\infty$, both in the non-homogeneous and the homogeneous case (it suffices to consider $\mathbf{K} = k(x_*)\mathbf{K}(\omega,\omega_*)$ with $\int_\Omega k = 1$). 
\item The global existence criterion offers direct global existence for non-normalized interactions $\alpha \neq 0$ but it also gives global existence for the original Vicsek equation with spatially-homogeneous kernel $\mathbf{K}(t,x,x_*,\omega,\omega_*) = \omega_*$ and fully homogeneous viscosity and negative friction (\textit{i.e.} only $f$ and time-dependent). Section \ref{subsec:global examples} describes several general cases where global existence happens in the problematic and purely normalised case $\alpha=0$.
\end{itemize}
\end{remark}
\bigskip

Besides the issue of well-posedness, it is of great interest to understand the large time behaviour of the solutions. We recall that Section \ref{subsec:global examples} proves global existence for different types of interactions and it also exhibits a free energy which decreases along the flow. One cannot expect a general theory since it heavily depends on the shape of the kernel $\mathbf{K}$ but one can still wonder if there are equilibria and if they are attractive. For this purpose, kinetic equations with gradient non-linearities  $\eqref{eq:general kinetic gradient}$ are often used  because one can extract an explicit free energy functional decreasing along time. Hence leading to the existence of equilibria and a hope for an asymptotic study of the solutions. This has been tackled in \cite{FrouLiu,DegFrouLiu1,DegFrouLiu2} where the authors exhibited a decreasing energy functional
\begin{equation}\label{eq:free energy homogeneous}
\mathcal{F}(t) = \int_{\S^{d-1}}f\mbox{ln}(f)d\omega + \frac{1}{2}\int_0^{\abs{\mathbf{J}_f}}\frac{\nu(s)}{\sigma(s)}ds
\end{equation}
when $\sigma$ and $\nu$ are solely functions of $\abs{\mathbf{J}_f}$. Associated to the latter is an energy dissipation allowing to dig out explicitely the equilibria in the spatially homogeneous setting. 
\par It appears that a free energy is also underlying in the general equation: $\eqref{eq:general kinetic orthogonal}$ enjoys a free energy functional that decreases along the flow with an explicit energy dissipation. Namely,
\begin{equation}\label{eq:free energy}
\begin{split}
\mathcal{F}[f](t) =& \int_{\Omega \times \S^{d-1}} f\mbox{ln}(f)dxd\omega
\\&+ \int_0^t \int_{\Omega \times \S^{d-1}} f \cro{\nabla_\omega\cdot\pa{\nu(f)\mathbf{P}_{\omega^\bot}[\boldsymbol\Psi[f]]} - \frac{\nu(f)^2}{\sigma(f)}\abs{\mathbf{P}_{\omega^\bot}[\boldsymbol\Psi[f]]}^2}dxd\omega ds
\end{split}
\end{equation}
and
\begin{equation}\label{eq:energy dissipation}
\mathcal{D}[f](t) = \int_{\Omega\times\S^{d-1}} \sigma(f) f \abs{\nabla_\omega \mbox{ln}(f) + \frac{\nu(f)}{\sigma(f)}\mathbf{P}_{\omega^\bot}[\boldsymbol\Psi[f]]}^2dxd\omega.
\end{equation}
which will be proven, in Section \ref{sec:long time behaviour}, to satisfy along the flow 
$$\frac{d}{dt}\mathcal{F}[f](t) = - \mathcal{D}[f](t).$$

\begin{remark}\label{rem:comparison free energy}
Let us emphasize that considering the spatially homogeneous case $\mathbf{K}= \omega_*$, $\eqref{eq:free energy}$ and $\eqref{eq:energy dissipation}$ are the ones obtained in \cite{DegFrouLiu1} with $\abs{\mathbf{J}_f}^2$ instead of $\abs{\mathbf{J}_f}$ (in case of constant viscosity and friction we also recover the original Doi-Onsager free energy \cite{Ons,Doi}). For general gradient kernel interactions $\psi[f]$ $\eqref{eq:general kinetic gradient}$ that are symmetric in $\omega$ and $\omega_*$ the free energy $\eqref{eq:free energy}$ becomes
$$\mathcal{F}(t) = \int_{\S^{d-1}}f\mbox{ln}(f)d\omega + \frac{1}{2}\int_0^{\langle\psi[f],f\rangle_{L^2_{x,\omega}}}\frac{\nu(s)}{\sigma(s)}ds$$
for $\nu$ and $\sigma$ being solely functions of $\langle\psi[f],f\rangle_{L^2_{x,\omega}}$ (which equals $\abs{\mathbf{J}_f}^2$). Hence offering a new view on the gradient structure where the natural dependencies are, in fact, $\langle\psi[f],f\rangle_{L^2_{x,\omega}} = \abs{\mathbf{J}_f}^2$.
\end{remark}

One can immediately see that the energy dissipation vanishes on 
\begin{equation}\label{eq:set of equilibria}
\mathcal{E}_\infty = \br{f_\infty \geq 0 \in L^\infty_tL^p_xH^2_\omega,\quad \nabla_\omega\pa{ \mbox{ln} f_\infty}(t,x,\omega) = -\frac{\nu(f_\infty)}{\sigma(f_\infty)}\mathbf{P}_{\omega^\bot}[\boldsymbol\Psi[f_\infty]](t,x,\omega) } 
\end{equation}
and when it vanishes so does the right-hand side of the kinetic equation $\eqref{eq:general kinetic orthogonal}$. The latter with the decrease of $\mathcal{F}[f](t)$ is a good hope that one could get a La Salle's invariance principle - in the spirit of \cite{FrouLiu} - when solutions are globally defined: namely the solution draws closer to the set of local equilibria $\mathcal{E}_\infty$.

\begin{remark}
If one establishes a La Salle's principle then necessarily, for a global equilibrium $f_\infty$ the quantity $\frac{\nu(f_\infty)}{\sigma(f_\infty)}\mathbf{P}_{\omega^\bot}[\boldsymbol\Psi[f_\infty]] = \nabla_\omega\pa{\psi[f]}$ must be a tangential gradient - and then $f_\infty(t,x,\omega) = A(t,x)e^{-\psi[f](t,x,\omega)}$ and we recover generalised von Mises equilibria of the previous \cite{FrouLiu,DegFrouLiu1,DegFrouLiu2}. Hence, as mentionned earlier, a gradient structure pops out naturally but does not solely concern $\mathbf{P}_{\omega^\bot}[\boldsymbol\Psi[f]]$ as first considered in $\eqref{eq:general kinetic gradient}$. It however is not the purpose of this article and we did not investigate further.
\end{remark}
\bigskip


\subsection{Microscopic point of view and mean-field limit for the Vicsek kinetic equation: strategy and results}
\label{subsec:particles intro}

We will prove that the kinetic equation can be obtained as the mean-field limit of some `approximated' (regularized) particle system. When there  is normalization ($\alpha = 0$), we will show that realizations of the approximated dynamics are also solutions of the Vicsek particle system with a probability which tends to one in the many-particle limit (Th. \ref{th:NormalizedCase}). Note, however, that for the case $\alpha=0$,  we will not prove the mean-field limit for the kinetic equation starting from particle systems of the form \eqref{eq:Vicsek_X}-\eqref{eq:Vicsek_omega}, nor its well-posedness. \\
We consider a system of $N$ particles given by their positions $X^{i,N}_t\in \Omega$ and velocities $V^{i,N}_t \in \R^d$ over time $t\geq 0$. Notice that we shall consider that the equations are written for $v\in \R^d$ rather than in $\omega \in \mathbb{S}^{d-1}$, but we shall later prove that the
velocities are restricted to the sphere with radius $c$, thus recovering the expected spherical dynamics.  We recall the empirical distribution of the process $(X^{i,N}_t,V^{i,N}_t)_{1\leq i \leq N}$:
\be \label{eq:empirical_distribution}
\mu^N_t(x,v) := \frac{1}{N}\sum_{i=1}^N\delta_{(X^{i,N}_t,V^{i,N}_t)}(x,v),
\ee
As mentioned before, in this article we will prove results on two types of particle systems. The first one we will call `general particle Vicsek' and the second one `approximated particle dynamics'. We define the \textbf{general particle Vicsek} as the particle system given by the following   Stratonovich stochastic differential equation:
\begin{subequations}\label{eq:particle_dyn_Vicsek}
\begin{numcases}{} 
dX^{i,N}_t = V^{i,N}_t dt,\\
dV^{i,N}_t = \nu(\mu^N_t)\, \mathbf{P}_{(V^{i,N}_t)^\perp} (\boldsymbol \Psi(X^{i,N}_t,V^{i,N}_t, \mu^N_t)) dt \nonumber \\
\qquad\quad+\frac{1}{2} \mathbf{P}_{(V^{i,N}_t)^\perp}[(\nabla_v\sigma(\mu^N_t))(X^{i,N}_t,V^{i,N}_t)]\, dt\label{eq:extra_term}\\
\qquad\quad + \sqrt{2\sigma(\mu^N_t)} \, \mathbf{P}_{(V^{i,N}_t)^\perp} \circ dB^i_t,\\
X^{i,N}_t(t=0)= \xin_0, \quad V^{i,N}_t(t=0)=\oin_0,
\end{numcases}
\end{subequations}
where  $\mathbf{P}_{v^\perp}$ is the projection operator
$$\mathbf{P}_{v^\perp}= \mbox{Id} - \frac{v\otimes v}{|v|^2},$$
where $\mbox{Id}$ is the identity matrix; $((B^i_t)_{t\geq })_{i=1,\hdots,N}$ are independent Brownian motions in $\R^d$. 
The symbol $'\circ'$ denotes that the stochastic differential equation \eqref{eq:particle_dyn_Vicsek} is in Stratonovich convention.\\ 
The precise way in which the function $\boldsymbol\Psi$ is extended to $V_t\in \R^d$ is explained in Sec. \ref{sec:non-singular-dynamics}, see system \eqref{eq:non-singular-particle_dyn}.

\bigskip
The `approximated particle system' is given by similar equations where the difference is that the functional $\boldsymbol \Psi$ is replaced by a functional $\tau_{\eps_0}$ that will be made precise later:
\begin{subequations}\label{eq:particle_dyn}
\begin{numcases}{} 
dX^{i,N}_t = V^{i,N}_t dt,\\
dV^{i,N}_t = \nu(\mu^N_t)\, \mathbf{P}_{(V^{i,N}_t)^\perp} (\tau_{\eps_0}(X^{i,N}_t,V^{i,N}_t, \mu^N_t)) dt \nonumber \\
\qquad\quad+\frac{1}{2} \mathbf{P}_{(V^{i,N}_t)^\perp}[(\nabla_v\sigma(\mu^N_t))(X^{i,N}_t,V^{i,N}_t)]\, dt\label{eq:extra_term}\\
\qquad\quad + \sqrt{2\sigma(\mu^N_t)} \, \mathbf{P}_{(V^{i,N}_t)^\perp} \circ dB^i_t,\\
X^{i,N}_t(t=0)= \xin_0, \quad V^{i,N}_t(t=0)=\oin_0,
\end{numcases}
\end{subequations}

The function $\tau_{\eps_0}=\tau_{\eps_0}(x,v,m)$ is defined such that 
\be \label{eq:taudefinition}
\tau_{\eps_0}(x,v, m)=\mathbf{\Psi}(x,v,m) \mbox{ whenever }|\mathbf J(x,v,m)|\geq \eps_0,
\ee
 for some $\eps_0>0$ and so that it is well defined in $\mathcal{H}$ (i.e., without singularities). In particular, in the system above we will have that
$$\tau_{\eps_0}(X^{i,N}_t,V^{i,N}_t, \mu^N_t)=\boldsymbol\Psi(X^{i,N}_t,V^{i,N}_t, \mu^N_t), \mbox{ if } |\mathbf J(X^{i,N}_t,V^{i,N}_t, \mu^N_t)|\geq \eps_0.$$
Because $\tau_{\eps_0}$ does not present any singularities when $|\mathbf J|=0$, the approximated particle system \eqref{eq:particle_dyn} can be thought of a regularization of the Vicsek particle system \eqref{eq:particle_dyn_Vicsek}.

Associated to the approximated particle system we define the approximated kinetic equation given by
\begin{equation}\label{eq:general kinetic approx}
\begin{split}
\partial_t f + c \omega\cdot\nabla_x f &= \nabla_\omega \cdot\pa{\sigma(f) \nabla_\omega f} +\nabla_\omega\cdot\pa{\nu(f)f\mathbf{P}_{\omega^\bot}\cro{\tau_{\eps_0}[f]}},
\end{split}
\end{equation}
Observe that  when $\tau_{\eps_0}$ is substituted by $\boldsymbol \Psi$ in the approximated kinetic equation \eqref{eq:general kinetic approx}, we obtain the Vicsek kinetic equation \eqref{eq:general kinetic orthogonal}.

\bigskip
First, we will show the well-posedness for the approximated particle system (Th. \ref{th:existence_solutions}) and that in the mean-field limit gives the Vicsek kinetic equation (for short times) (Th. \ref{th:convergence_process} and Cor. \ref{cor:mean-field}).

\begin{remark}
Notice that the term \eqref{eq:extra_term} in system \eqref{eq:particle_dyn_Vicsek} appears so that we obtain the kinetic equation \eqref{eq:general kinetic orthogonal}. This is just a technicality: system \eqref{eq:particle_dyn_Vicsek} in  It\^o's convention corresponds to:
\begin{subequations}\label{eq:particle_dyn_Ito}
\begin{numcases}{} 
dX^{i,N} = V^{i,N} dt,\\
dV^{i,N} = \nu(\mu^N)\, P_{(V^{i,N})^\perp} (\boldsymbol \Psi(X^{i,N},V^{i,N}, \mu^N)) dt \nonumber \\
\qquad\quad+ P_{(\boin)^\perp}[(\nabla_v\sigma(\mu^N))(\bxin,\boin)]\, dt\nonumber\\
\qquad\quad + \sqrt{2\sigma(\mu^N)} \, P_{(V^{i,N})^\perp} dB^i_t\nonumber\\
\qquad\quad- 2(d-1) \sigma(\mu^N)\frac{V^{i,N}}{|V^{i,N}|^2},\\
\xin(t=0)= \xin_0, \quad \oin(t=0)=\oin_0.
\end{numcases}
\end{subequations}
 See also Appendix \ref{sec:Ito_conversion} for more details.  Without the extra term \eqref{eq:extra_term} we would obtain a kinetic equation where the operator in $\omega$ is of the form (see Eq. \eqref{eq:ito_application})
$$\Delta_\omega (\sigma(f) f)-\frac{1}{2}\nabla_\omega \cdot (\nabla_\omega \sigma(f) f).$$
But with the extra term the operator $\omega$ in the kinetic equation is of the form
$$\Delta_\omega(\sigma(f) f) -\nabla_\omega\cdot(\nabla_\omega\sigma(f) f) =\nabla_\omega\cdot (\sigma(f) \nabla_\omega f),$$
which is the operator that we are dealing with in \eqref{eq:general kinetic orthogonal}. Notice that if we wanted to obtain just a factor $\Delta_\omega (\sigma(f) f)$ then the constant in front of the extra term \eqref{eq:extra_term} should be $-1/2$ rather than $1/2$. Notice also that if $\sigma=\sigma(t,x,f)$ does not depend on $\omega$, then the extra term \eqref{eq:extra_term} does not appear. This is the case, for example, when $\sigma=\sigma(|\tilde{\mathbf{J}}_f|(t,x))$ for $\tilde{\mathbf{J}}_f$ defined in \eqref{eq:aux_example_flux}, see \cite{DegFrouLiu2}. The extra term \eqref{eq:extra_term} has the effect of relaxing $V^{i,N}$ towards the value taken by $\nabla_v \sigma$. It can be rewritten equivalently as a gradient
$$\frac{1}{2} \mathbf{P}_{(V^{i,N}_t)^\perp}[(\nabla_v\sigma(\mu^N_t))(X^{i,N}_t,V^{i,N}_t)]\, dt=\frac{1}{2}\nabla_{V^{i,N}}(V^{i,N}\cdot \nabla_v\sigma(\mu^N_t))(X^{i,N}_t,V^{i,N}_t),$$
where the gradient is on the sphere.

\end{remark}
\bigskip

As announced before, the dynamics described by our agent-based model indeed force the velocities to have unit norm, as shown by the next lemma.
\begin{lemma}
\label{lem:norm_1_particles}
Suppose that $(\xin_t,\oin_t)$ is a solution to the approximated particle dynamics \eqref{eq:particle_dyn} or the Vicsek particle dynamics \eqref{eq:particle_dyn_Vicsek}. Suppose that $|\oin_0|=1$ for all $i=1,\hdots, N$, then it holds that
$$|\oin_t| = 1, \qquad \mbox{ for all } i=1,\hdots, N,$$
for all times where the solution is defined.
\end{lemma}
\begin{proof}[Proof of Lemma \ref{lem:norm_1_particles}]
It is a direct check that
$$d|\oin_t|^2=0$$
using Stratonovich chain rule (see, for example, \cite[p. 122]{evans2012introduction}) and the fact that $V\cdot \mathbf{P}_{V^\perp}=0$.
\end{proof}

The potential degeneracy of $\boldsymbol\Psi$ breaks the Lipschitz regularity of the interactions and thus prevents standard agent-based well-posedness results to apply. There exist results for non-Lipschitz interactions \cite{BolCanCar1} but their singularities differ from the one at stake in the present article.
\par  In the spirit of \cite{BolCanCar2}, we thus shall prove well-posedness of the system hand-in-hand with the mean-field limit with the aid of an auxiliary process which mixes microscopic dynamics with the mesoscopic distribution function $f_t$. The main difference with respect to \cite{BolCanCar2} is that we need to deal the fact that the noise coefficient $\sigma$ is not constant. Let us explicit the auxiliary process for the approximated particle system $(\bxin_t, \boin_t)_{t\geq 0}$, $i=1,\hdots, N$, solution of
\begin{subequations} \label{eq:auxiliary}
\begin{numcases}{}
d\bxin_t = \boin_t dt,\\
d\boin_t = \nu(f_t)\, \mathbf{P}_{(\boin_t)^\perp} (\tau_{\eps_0}(\bxin_t,\boin_t,f_t)) dt \nonumber\\
\qquad\quad+\frac{1}{2} \mathbf{P}_{(\boin_t)^\perp}[(\nabla_\omega\sigma(f))(\bxin_t,\boin_t)]\, dt\nonumber\\
\qquad\quad+ \sqrt{2\sigma(f_t)}\, \mathbf{P}_{(\boin_t)^\perp} \circ dB^i_t,\\
f_t = \mbox{law}(\bxin_t, \boin_t),\\
\bxin_t(t=0) = \xin_0, \quad \boin_t(t=0)=\oin_0,
\end{numcases}
\end{subequations}
where the initial data and the Brownian processes $(B^i_t)_{t\geq 0}$ are the same ones as in \eqref{eq:particle_dyn}. Notice that all the processes $(\bxin_t, \boin_t)$ are independent by construction and they all have the same law $f_t$.

\bigskip
We will prove the following two results for the approximated particle system \eqref{eq:particle_dyn} and the approximated kinetic equation \eqref{eq:general kinetic approx}.

\bigskip
\begin{theorem}[Local-in-time existence and uniqueness of solutions for the approximated system]\label{th:existence_solutions}

Under the assumptions of Theorem \ref{theo:local cauchy} and assumption (H4),
let $f_0$ be a probability measure on $\R^d \times \mathbb{S}^{d-1}$  with finite second moment in $x\in \R^d$ and let $(X_{0}^{i,N}, V^{i,N}_0)$ for $i=1,\hdots, N$ be $N$ independent random variables with law $f_0$. The following holds:
\begin{itemize}
\item[(i)]  There exists a pathwise global unique solution to the SDE system \eqref{eq:particle_dyn}  with initial data $(X_{0}^{i,N}, V^{i,N}_0)$ for $i=1,\hdots, N$. Moreover, the solution is such that $|V^{i,N}_t|=1$.
\item[(ii)] There exists a pathwise global unique solution to the auxiliary process \eqref{eq:auxiliary} with initial data $(X_{0}^{i,N}, V^{i,N}_0)$ for $i=1,\hdots, N$ and $|\boin_t|=1$.
\item[(iii)] There exists a global-in-time unique weak solution of the approximated kinetic equation \eqref{eq:general kinetic approx} with initial datum $f_0$. The solution of the kinetic equation is the law of the process solution to the auxiliary System \eqref{eq:auxiliary}, wherever the solution is defined. 
\end{itemize}
\end{theorem}
\bigskip

Along with this well-posedness result, we rigorously show its  mean-field limit towards the approximated kinetic equation \eqref{eq:general kinetic approx}.

\bigskip
\begin{theorem}[Mean-field limit for the approximated system - propagation of chaos]
\label{th:convergence_process}
Under the assumptions of Theorem \ref{theo:local cauchy} and assumption (H4), for the respective solutions $(\xin_t,\oin_t)_{t\geq 0}$ and $(\bxin_t, \boin_t)_{t\geq 0}$ of \eqref{eq:particle_dyn} and \eqref{eq:auxiliary},  for any $T>0$, it holds
\be \label{eq:limit_particles} 
\lim_{N\to\infty}\sup_{{1\leq i\leq N}}\mathbb{E}\big[\sup_{0\leq t\leq T} \big(|\xin_t-\bxin_t|^2 + |\oin_t - \boin_t|^2\big) \big] =0.
\ee
\end{theorem}

From this we deduce a mean-field limit result for the Vicsek kinetic equation:
\begin{cor}[Mean-field limit for the Vicsek kinetic equation]
\label{cor:mean-field}
Suppose that the assumptions in Th. \ref{theo:local cauchy} and assumption (H4) hold. Let $f_t$ be the local-in-time solution of \eqref{eq:general kinetic orthogonal} given by Th. \ref{theo:local cauchy} for $t\in [0,T)$. Then, there exists and $\eps_0>0$ such that the law of the auxiliary process \eqref{eq:auxiliary} is precisely $f_t$ for any $t\in[0,T)$. Consequently, Th. \ref{th:convergence_process} holds for $(\bar X^{i}_t, \bar V^{i}_t)$ having law $f_t$ solution to the kinetic Vicsek equation \eqref{eq:general kinetic orthogonal}.
\end{cor}
\begin{proof}
Take $T_0 \in [0,T)$ and define
$$\eps_0 := \inf_{t\in[0, T_0]} \inf_{(x,v)}|\mathbf{J}(t,x,v,f_t)|>0.$$
Then, with this value of $\eps_0$, the approximated kinetic equation \eqref{eq:general kinetic approx} coincides with the Vicsek kinetic equation \eqref{eq:general kinetic orthogonal} for $t\in[0,T_0]$. Therefore,  by Th. \ref{th:existence_solutions},   the solution of the Vicsek kinetic equation $f_t$ is the law of the auxiliary process \eqref{eq:auxiliary} for $t\in [0,T_0]$ and Th. \ref{th:convergence_process} applies.
\end{proof}
Cor. \ref{cor:mean-field} implies the mean-field limit of the approximated dynamics \eqref{eq:particle_dyn} towards the Vicsek kinetic equation \eqref{eq:general kinetic orthogonal} for short times (in the case $\alpha\neq 0$, then $T=+\infty$ and the mean-field limit holds for all times).

\bigskip
\begin{remark}[Convergence of the measures]
Expression \eqref{eq:limit_particles} ensures the convergence as $N\to\infty$ of the law of the process $(X^{i,N}_t, V^{i,N}_t)$ towards $f_t$ for any $i$ and $t\in [0,T]$. See the notes after Th. 1.1 in \cite{BolCanCar1} for more details. In particular, in \cite{BolCanCar1} it is shown the following upper bound in 2-Wasserstein distance
$$W_2^2(f^{(1)}_t, f_t) \leq \mathbb{E}\left[ |X^{i}_t - \bar X^{i}_t|^2 + |V^{i}_t - \bar V^{i}_t|^2\right] \leq \eps(N),$$
where $\eps(N)\to 0$ as $N\to \infty$ (by Th. \ref{th:convergence_process}). The function $f^{(1)}$ denotes the first marginal of the $N$ particle system. They also show that for any  Lipschitz map $\varphi$
\be \label{eq:l2conv}
\mathbb{E}\left[\left| \frac{1}{N}\sum^N_{i=1} \varphi(X^i_t, V^i_t) - \int \varphi \, df_t \right|^2 \right] \leq \eps(N) + \frac{C}{N}
\ee
for some constant $C>0$ independent of $N$.
\\
Under more regularity assumptions, one can obtain explicit estimates on $\eps(N)$. See Th. 10 in Ref. \cite{carmona2016lectures} (Sec. 1.3.4).
\end{remark}

\subsubsection{The Vicsek particle system when $\alpha=0$}

All the previous results for the mean-field limit correspond only to the approximated system \eqref{eq:particle_dyn}, which is not singular when $\alpha =0$. When the norm of the flux 
\be \label{eq:desiredestimate}
|\mathbf J(t, X^{i,N}, V^{i,N}, \mu^N_t)|~\geq ~\eps_0,
\ee
for all $t\in[0,T]$ and all $i=1,\hdots, N$,
 the approximated particle dynamics \eqref{eq:particle_dyn} coincides with the Vicsek particle dynamics \eqref{eq:particle_dyn_Vicsek} for $t\in[0,T]$.

When $\alpha =0$ (normalized case),  every realization of the `approximated' particle system \eqref{eq:particle_dyn} such that $|\mathbf J(t, X^{i,N}_t, V^{i,N}_t, \mu^N_t)|>\eps_0$ for all $t \in [0,T]$ and all $i=1,\hdots N$ is also a solution to the Vicsek particle system. The next result shows that for short times \eqref{eq:desiredestimate} happens  with a probability which tends to one in the many-particle limit:
\begin{theorem}[Lower bound on the flux for the normalized case]
\label{th:NormalizedCase}
Suppose that we are under the assumptions of Th. \ref{theo:local cauchy} and assumption (H4). Consider a time $T<T_1$ (where $T_1$ is defined in \eqref{T1}) and $c_*=c_*(T)>0$ given by \eqref{eq:defcstart}. Then, for all $t\in [0,T]$ and all $\eps_0\in (0, c_*)$ it holds that
\be \label{eq:highprob}
\mathbb{P}(\inf_{t\in[0,T]}\inf_{(x,v)}|\mathbf J(t,x,v, \mu^N_t)|>\eps_0) \geq 1- \frac{1}{c_*-\eps_0}\, \eps(N),
\ee
 where $\eps(N)\to 0$ as $N\to \infty$. 
\end{theorem}

Define $A_{T,\eps_0}^{(N)}$ as the event such that:
$$A_{T,\eps_0}^{(N)}=\{\mbox{for all $t\in[0,T]$ and all $i=1,\hdots, N$, $\left| \boldsymbol J(t,X^{i,N}_t, V^{i,N}_t, \mu^N_t)\right| >\eps_0$}\}.$$
 Then, as a consequence of Th. \ref{th:NormalizedCase}, it holds that for all $\eps_0 \in (0,c_*)$
\be \label{eq:boundParticles}
\mathbb P \lp A_{T,\eps_0}^{(N)} \rp \geq 1- \frac{1}{c_*-\eps_0}\eps(N), 
\ee
and so with a probability which tends to one in the many-particle limit, in the sense of \eqref{eq:boundParticles}, realizations of the approximated particle dynamics \eqref{eq:particle_dyn} will have non-zero flux for short times and they will also be solutions to the general Vicsek particle system \eqref{eq:particle_dyn_Vicsek}.

%

\begin{remark}
Proving the well-posedness and the rigorous mean-field limit for the Vicsek particle system \eqref{eq:particle_dyn_Vicsek} in the normalized case ($\alpha =0$) seems to at least require (using the current strategy)  to show
$$|\mathbf J(t,x,v, \mu^N_t)|>0 \mbox{ a.s.,}$$
for all $N$ large enough, for all times in some interval and all values of $(x,v)$. It is not clear if this is even true. This would require some kind of generalization of a strong law of large numbers. Such strong law of large numbers exists for infinite exchangeable sequences as a result of de Finetti's theorem (see \cite{kingman1978uses}). However, to the best of our knowledge, there is not an analogous result for a hierarchy of finite exchangeable sequences. 

\end{remark}

\section{The mesoscopic framework: kinetic differential equation}\label{sec:kinetic}

\subsection{Linear equation: dependence on coefficients and averaging positivity}\label{sec:linear fokker planck}

The differential operator $\nabla_\omega$ is the tangential gradient - also called G\"unter derivatives - on the sphere $\S^{d-1}$. It can be related to the standard gradient on $\R^d$ by
$$\nabla_\omega = \nabla_v - \langle \nabla_v, \omega \rangle_{\R^d}\: \omega = \mathbf{P}_{\omega^\bot}\cro{\nabla_v}$$
where $\nabla_v$ is the euclidean gradient for functions from $\R^d$ to $\R^d$. First of all let us emphasize that the tangential gradient $\nabla_\omega$ displays very different behaviour than the usual gradient. We give here three formulas that we shall use throughout the proofs. We refer the interested reader to \cite{Dud,BenTor} for an introduction on tangential - G\"unter- derivatives and to \cite[Appendix II]{OttTza} for the specific calculus of the following formulas (proven in dimension $3$ but immediately extendable in dimension $d$). Integration by parts is allowed but generate an additional term in the direction of the orientation $\omega$:
\begin{equation}\label{eq:IBP tangential derivative}
\int_{\S^{d-1}}\nabla_\omega \pa{f(\omega)}g(\omega)d\omega = - \int_{\S^{d-1}}f(\omega) \nabla_\omega \pa{g(\omega)}d\omega + (d-1)\int_{\S^{d-1}} \omega\, fg \:d\omega.
\end{equation}
It also holds the following explicit tangential derivatives for each coordinate $\omega_i$ of $\omega= (\omega_1,\dots,\omega_d)$ on the sphere
\begin{equation}\label{eq:tangential derivative omega}
\nabla_\omega\pa{\omega_i} = \pa{\delta_{ij} - \omega_i\omega_j}_{1\leq j \leq d} \quad\mbox{and}\quad \Delta_\omega(\omega_i) = -(d-1)\omega_i
\end{equation}

\bigskip
As mentioned in the introduction, $\eqref{eq:general kinetic orthogonal}$ can be viewed as a nonlinear Fokker-Planck equation on the torus with local and nonlinear viscosity $\sigma$ and drift $\nu$. The purpose of this section is to study the associated linear equation when the nonlinearities are considered as given parameters. In other words we consider here the following differential problem

\begin{equation}\label{eq:linear Fokker-Planck}
\partial_t f + c \omega\cdot\nabla_x f = \nabla_\omega\cdot\pa{\bar{\sigma}(t,x,\omega)\nabla_\omega f + f \bar{\boldsymbol\Psi} (t,x,\omega)}
\end{equation}
where $\bar{\sigma}$ and $\bar{\boldsymbol\Psi}$ are given functions.
\par The issue of existence and uniqueness for $\eqref{eq:linear Fokker-Planck}$ has been solved for instance in \cite[Appendix A]{Deg} for velocities in $\R$ or $\R^2$ for constant $\bar{\sigma}$ with $\bar{\boldsymbol\Psi}$ and $\nabla_\omega \cdot \bar{\boldsymbol\Psi}$ being in $L^\infty_{x,\omega}$. However their methods are directly applicable for velocities in $\S^{d-1}$ as shown in \cite[Lemma 4.1]{GamKan} for constant $\bar{\sigma}$. The case of non-constant $\bar{\sigma}$ is a straightforward adaptation of their proofs if $\bar{\sigma}$ belongs to $L^\infty_{t,x,\omega}$ and is uniformly bounded from below by a constant $\bar{\sigma}_0 >0$ (note that one could also adapt to tangential derivatives standard Galerkin type methods for linear parabolic equations \cite[Chapter 6]{Eva}, as done in \cite{FrouLiu} in the spatially homogeneous case). We thus omit the proof and state:

\bigskip
\begin{theorem}\label{theo:cauchy linear}
Let $\bar{\sigma}$ be in $L^\infty_{t,x,\omega}$ uniformly bounded from below: $\bar{\sigma}(t,x,\omega) \geq \bar{\sigma}_0>0$ and let $\bar{\boldsymbol\Psi}$ and $\nabla_\omega \cdot \bar{\boldsymbol\Psi}$ be in $L^\infty_{t,x,\omega}$ with the orthogonal property $\mathbf{P}_{\omega^\bot}(\bar{\boldsymbol\Psi})=\bar{\boldsymbol\Psi}$.
\\ Then for any $f_0$ in $L^2_{x,\omega}$ there exists a unique $f$ in the space
$$Y = \br{u \in L^2_{t,x}H^1_\omega, \quad \partial_t u + \omega \cdot \nabla_x u \in L^2_{t,x}H^{-1}_\omega}$$
solution to $\eqref{eq:linear Fokker-Planck}$ with initial datum $f_0$.
\\Moreover, if $f_0$ is non-negative then $f$ is non-negative.
\end{theorem}
\bigskip

The next Section \ref{subsec:gain regularity} is dedicated to $L^p$ bounds and gain of regularity for solutions to $\eqref{eq:linear Fokker-Planck}$ whilst Section \ref{subsec:dependence coef} studies the dependences on the viscosity and friction. To conclude, Section \ref{subsec:averaging positivity} tackles the issue of an explicit lower bound on the vanishing time for velocity averaging densities.
\bigskip


\subsubsection{$L^p$ bounds and gain of regularity}\label{subsec:gain regularity}

The issue of $L^p_{x,\omega}$ estimates and gain of regularity had already been investigated in previous cited references for particular $\bar{\sigma}$ and $\bar{\boldsymbol\Psi}$ and we provide here an adapted version that fits our general setting.

\bigskip
\begin{prop}\label{prop:Lp linear}
Let $p$ belong to $[2,+\infty]$ and $f_0\geq 0$ be in $L^1_{x,\omega}\cap L^p_{x,\omega}$ with mass $\norm{f_0}_{L^1_{x,\omega}}=M_0$. Under the assumptions of Theorem \ref{theo:cauchy linear} on $\bar{\sigma}$ and $\bar{\boldsymbol\Psi}$ the solution $f$ to $\eqref{eq:linear Fokker-Planck}$ is in $L^p_{x,\omega}$ and  preserves the mass $M_0$. More precisely it satisfies
$$\forall t\geq 0,\quad \norm{f(t)}_{L^p_{x,\omega}} \leq e^{C_p(\bar{\sigma},\bar{\boldsymbol\Psi})t} \norm{f_0}_{L^p_{x,\omega}},$$
and  it also gains regularity in $\omega$ when $p<+\infty$ in the following sense:
$$\int_0^te^{C_p(\bar{\sigma},\bar{\boldsymbol\Psi})(t-s)}\norm{f^{\frac{p-2}{2}}\nabla_\omega f(s)}^2_{L^2_{x,\omega}}\norm{f(s)}^{1-p}_{L^p_{x,\omega}}ds \leq \frac{2\norm{f_0}_{L^p_{x,\omega}}}{(p-1)\bar{\sigma}_0}e^{C_p(\bar{\sigma},\bar{\boldsymbol\Psi})t},$$
where 
\begin{equation}\label{def:Cp}
C_p(\bar{\sigma},\bar{\boldsymbol\Psi}) =\norm{\nabla_\omega\cdot\bar{\boldsymbol{\Psi}}}_{L^\infty_{t,x,\omega}}+\frac{\norm{\bar{\boldsymbol\Psi}}^2_{L^\infty_{t,x,\omega}}}{(p-1)\bar{\sigma}_0},
\end{equation}
with $C_\infty = \lim_{p\to\infty} C_p$.
\end{prop}
\bigskip

\begin{proof}[Proof of Proposition \ref{prop:Lp linear}]
\textbf{Step 1: $L^p$ bounds and gain of regularity.} Note that since $\nabla_\omega$ and $\bar {\boldsymbol{\Psi}}$ are orthogonal to $\omega$, we can perform  on $\eqref{eq:linear Fokker-Planck}$ integration by parts on $\S^{d-1}$ as if we were working with classical derivatives. Indeed, the orthogonality to $\omega$ exactly cancels the additional term in the tangential derivatives integration by parts formula $\eqref{eq:IBP tangential derivative}$. It yields for any $p\geq 2$,
\begin{equation*}
\begin{split}
I:=\frac{1}{p}&\frac{d}{dt}\norm{f(t)}^p_{L^p_{x,\omega}} = \int_{\Omega  \times \S^{d-1}} \mbox{sgn}(f(t,x,v))\abs{f(t,x,v)}^{p-1}\partial_t f(t,x,\omega)\:dxd\omega
\\&= -  \int_{\Omega \times\S^{d-1}}\mbox{sgn}(f)\abs{f}^{p-1}\nabla_x(c\omega f) -\int_{\Omega \times \S^{d-1}} \bar{\sigma}(t,x,\omega)\nabla_\omega \pa{\mbox{sgn}(f)\abs{f}^{p-1}}\cdot\nabla_\omega f
\\&\quad+\int_{\Omega \times \S^{d-1}} \mbox{sgn}(f)\abs{f}^{p-1} \nabla_\omega\cdot \pa{f \bar{\boldsymbol{\Psi}}}\:dxd\omega
\end{split}
\end{equation*}
We use that the first integrand on the right-hand side can be written as a divergence in $x$ : $\nabla_x\cdot (c\omega \abs{f}^p)$ whereas the second one can be written as $(p-1)\bar{\sigma}(t,x,\omega)\abs{f}^{p-2} \pa{\nabla_\omega f}^2$. This leads to the following upper bound
\begin{equation*}
\begin{split}
I\leq& - (p-1)\bar{\sigma}_0\norm{f^{\frac{p-2}{2}}\nabla_\omega f}_{L^2_{x,\omega}}^2 + \int_{\Omega\times\S^{d-1}} \pa{\abs{f}^p\nabla_\omega\cdot\bar{\boldsymbol\Psi}+ \bar{\boldsymbol\Psi} \abs{f}^{p-1}\abs{\nabla_\omega f}}dxd\omega
\\\leq& - (p-1)\bar{\sigma}_0\norm{f^{\frac{p-2}{2}}\nabla_\omega f}_{L^2_{x,\omega}}^2 +\norm{\nabla_\omega\cdot\bar{\boldsymbol{\Psi}}}_{L^\infty_{t,x,\omega}}\norm{f}^p_{L^p_{x,\omega}}+ \norm{\bar{\boldsymbol{\Psi}}}_{L^\infty_{t,x,\omega}}\norm{f}^{\frac{p}{2}}_{L^p_{x,\omega}}\norm{f^{\frac{p-2}{2}}\nabla_\omega f}_{L^2_{x,\omega}}
\\\leq& - \pa{(p-1)\bar{\sigma}_0 - \frac{\eta}{2}\norm{\bar{\boldsymbol{\Psi}}}^2_{L^\infty_{t,x,\omega}}}\norm{f^{\frac{p-2}{2}}\nabla_\omega f}_{L^2_{x,\omega}}^2 + \pa{\norm{\nabla_\omega\cdot\bar{\boldsymbol{\Psi}}}_{L^\infty_{t,x,\omega}}+\eta^{-1}}\norm{f}^p_{L^p_{x,\omega}}.
\end{split}
\end{equation*}
Note that we used Cauchy-Schwarz and Young inequality for any $\eta >0$. Choosing $\frac{p-1}{2}\bar{\sigma}_0 - \frac{\eta}{2}\norm{\bar{\boldsymbol\Psi}}^2_{L^\infty_{t,x,v}} = 0$ end up with
\begin{eqnarray*}
\frac{d}{dt}\norm{f(t)}_{L^p_{x,\omega}}&=& \frac{1}{p} \frac{d}{dt}\norm{f(t)}^p_{L^p_{x,\omega}}\norm{f(t)}^{1-p}_{L^p_{x,\omega}}
\\&\leq& -\frac{(p-1)}{2}\bar{\sigma}_0\norm{f^{\frac{p-2}{2}}\nabla_\omega f}_{L^2_{x,\omega}}^2\norm{f}^{1-p}_{L^p_{x,\omega}} + C_p\norm{f}_{L^p_{x,\omega}}.
\end{eqnarray*}
From which we deduce thanks to Gr\"onwall lemma that for all $t\geq 0$
\begin{equation}\label{gain regularity finish}
\norm{f(t)}_{L^p_{x,\omega}}+ \frac{(p-1)}{2}\bar{\sigma}_0\int_0^t e^{C_p(t-s)} \norm{f^{\frac{p-2}{2}}(s)\nabla_\omega f(s)}^2_{L^2_{x,\omega}}\norm{f(s)}^{1-p}_{L^p_{x,\omega}}ds\leq e^{C_p t}\norm{f_0}_{L^p_{x,\omega}}.
\end{equation}
The latter is exactly the standard bounds in Proposition \ref{prop:Lp linear} for $p$ in $[2,+\infty)$, obtained as an \textit{a priori} estimate. The belonging to $L^p$ for $f$ follows standard methods (also used in \cite{GamKan}) where our kinetic equation is approximated by a linear iterative scheme for which each solution $f_n$ is in $L^p$ and \eqref{gain regularity finish} holds at each step. Then the uniqueness of the kinetic solution implies that $f_n\rightarrow f$ and taking the limit in \eqref{gain regularity finish} shows that $f$ belongs to $L^p$ and satisfies \eqref{gain regularity finish}.

\bigskip
\textbf{Step 2: Mass conservation and $L^\infty$ bounds.}
Now let us suppose that $f_0$ is non-negative, which implies that $f$ is non-negative at all time by Theorem  \ref{theo:cauchy linear}. Hence, it holds that
\begin{eqnarray*}
\frac{d}{dt}\norm{f(t)}_{L^1_{x,\omega}} &=& \frac{d}{dt}\int_{\Omega\times \S^{d-1}} f(t,x,\omega)\:dxd\omega 
\\&=& -\int_{\Omega\times\S^{d-1}}\nabla_x\cdot(c\omega f) + \nabla_\omega\cdot\pa{\bar{\sigma}\nabla_\omega f + f \bar{\boldsymbol\Psi}}dxd\omega = 0
\end{eqnarray*}
because, again, the integration by parts $\eqref{eq:IBP tangential derivative}$ does not generate additional terms in the direction of $\omega$. This concludes the preservation of the $L^1_{x,\omega}$-norm.
\par The $L^\infty_{x,\omega}$ estimates follows straight from the limit $p$ goes to $+\infty$, since $f(t)$ belongs to $L^1_{x,\omega}$.
\end{proof}
\bigskip


\subsubsection{Dependence on the coefficients}\label{subsec:cauchy linear}\label{subsec:dependence coef}

Our strategy to tackle the non-linear equation is \textit{via} a fixed point argument. We thus need to understand how solutions to $\eqref{eq:linear Fokker-Planck}$ differ from one another when they are associated to different coefficients $\bar{\sigma}$ and $\bar{\boldsymbol\Psi}$. The main issue in establishing estimate on $\norm{f_1-f_2}_{L^p_{x,\omega}}$ of two different solutions relies on the fact that the gain of regularity proved in Proposition \ref{prop:Lp linear} is higly non-linear as soon as $p>2$. Moreover, we did not manage to quantify the way the gain of regularity vanishes at initial time so we cannot close a direct $L^\infty_t$ fixed point method. We shall thus only study the dependence on coefficient in $L^2_{t,x}$.

\bigskip
\begin{prop}\label{prop:dependence coef}
Let $f_0$ be a non-negative function in $L^1_{x,\omega}\cap L^2_{x,\omega}$. Suppose that $\bar{\sigma}_1$, $\bar{\sigma}_2$ and $\bar{\boldsymbol\Psi}_1$, $\bar{\boldsymbol\Psi}_2$ satisfy the assumptions of Theorem \ref{theo:cauchy linear}. Let $f_1$ and $f_2$ be the solutions of $\eqref{eq:linear Fokker-Planck}$ associated respectively to the coefficient $(\bar{\sigma}_1,\bar{\boldsymbol\Psi}_1)$ and $(\bar{\sigma}_2,\bar{\boldsymbol\Psi}_2)$ with initial datum $f_0$.
\\Then the following holds for any $t\geq 0$,
\begin{equation*}
\begin{split}
\norm{(f_1-f_2)(t)}_{L^2_{x,\omega}}^2 \leq& \frac{4e^{3C_2t}\norm{f_0}^2_{L^2_{x,\omega}}}{\inf\br{\bar{\sigma_1}+\bar{\sigma_2}}}\int_0^t\norm{\pa{\bar{\boldsymbol\Psi}_1-\bar{\boldsymbol\Psi}_2}(s)}^2_{L^\infty_{x,\omega}}\:ds
\\&+\frac{2}{\inf\br{\bar{\sigma_1}+\bar{\sigma_2}}}\cro{\int_0^t e^{C_2(t-s)}\pa{\norm{\nabla_\omega f_1}^2_{L^2_{x,\omega}}+\norm{\nabla_\omega f_2}^2_{L^2_{x,\omega}}}ds}\norm{\bar{\sigma}_1-\bar{\sigma}_2}_{L^\infty_{t,x,\omega}}^2,
\end{split}
\end{equation*}
where $C_2 = C_2(\bar{\sigma_1},\bar{\boldsymbol\Psi_1})+C_2(\bar{\sigma_2},\bar{\boldsymbol\Psi_2})$ defined in Proposition \ref{prop:Lp linear}.
\end{prop}
\bigskip

\begin{remark}\label{rem:dependence coef}
The main difficulty with dealing with local-in-time $\sigma(f)$ appears in the second term on the right-hand side above. Indeed, from Proposition \ref{prop:Lp linear} we see that $\norm{\nabla_\omega f_i}^2_{L^2_{x,\omega}}$ is integrable on $[0,t]$ and thus $\int_0^t\norm{\nabla_\omega f_i}^2_{L^2_{x,\omega}}$ vanishes when $t$ goes to $0$, thus almost allowing to close a contraction fixed point argument. As we shall see later on, we however need an explicit independence over $\bar{\sigma}$ (other than $\inf\bar{\sigma}$ and $\sup\bar{\sigma}$) to avoid any nonlinear breakdown of contraction theorem in short times. As a result a $L^2([0,T],L^2_{x,\omega})$ setting is required to get an extra integration in time. This explains hypothesis $(H2)$, but one could also ask for more regularity for $\sigma$ and $f_0$ to explicitly estimate the convergence to $0$ of $\int_0^t\norm{\nabla_\omega f_i}^2_{L^2_{x,\omega}}$ with parabolic regularity.
\end{remark}

\begin{proof}[Proof of Proposition \ref{prop:dependence coef}] To shorten notations we will use $\bar{\sigma}^+ = \bar{\sigma}_1+\bar{\sigma}_2$, $\bar{\sigma}^- = \bar{\sigma}_1-\bar{\sigma}_2$, $\bar{\boldsymbol\Psi}^+= \bar{\boldsymbol\Psi}_1+\bar{\boldsymbol\Psi}_2$ and $\bar{\boldsymbol\Psi}^-= \bar{\boldsymbol\Psi}_1-\bar{\boldsymbol\Psi}_2$. Also we denote the constant constructed from $\eqref{def:Cp}$ by $C_2 = \max\br{C_2(\bar{\sigma}_1,\bar{\boldsymbol\Psi}_1),C_2(\bar{\sigma}_2,\bar{\boldsymbol\Psi}_2),C_2(\bar{\sigma}^+,\bar{\boldsymbol\Psi}^+),C_2(\bar{\sigma}^-,\bar{\boldsymbol\Psi}^-)}$.
\par Using the algebraic identity $ab-cd = \frac{1}{2}[(a-c)(b+d) + (a+c)(b-d)]$ we find
\begin{equation*}
\begin{split}
\partial_t\cro{f_1-f_2} +& c \omega \nabla_x\cdot\cro{f_1-f_2}
\\=& \frac{1}{2} \nabla_\omega\cdot\pa{\pa{\bar{\sigma}_1 + \bar{\sigma}_2}\nabla_\omega\cro{f_1-f_2}} + \frac{1}{2} \nabla_\omega\cdot\pa{\cro{f_1-f_2}\pa{\bar{\boldsymbol\Psi}_1 + \bar{\boldsymbol\Psi}_2}} 
\\&+ \frac{1}{2} \nabla_\omega\cdot\cro{\pa{\bar{\sigma}_1 - \bar{\sigma}_2}\nabla_\omega\pa{f_1+f_2}} + \frac{1}{2} \nabla_\omega\cdot\cro{\pa{f_1+f_2}\pa{\bar{\boldsymbol\Psi}_1 - \bar{\boldsymbol\Psi}_2}},
\end{split}
\end{equation*}
We start by bounding the terms in $f_1-f_2$ as in Proposition \ref{prop:Lp linear} and we obtain
\begin{equation}\label{dependence start}
\begin{split}
\frac{1}{2}\frac{d}{dt}\norm{f_1-f_2}^2_{L^2_{x,\omega}}\leq& -\frac{\bar{\sigma}^+_0}{4}\norm{\nabla_\omega(f_1-f_2)}^2_{L^2_{x,\omega}}+ C_2\norm{f_1-f_2}^2_{L^2_{x,\omega}}
\\&+\frac{1}{2}\int_{\Omega\times\S^{d-1}} (f_1-f_2)\nabla_\omega\cdot\cro{\bar{\sigma}^-\nabla_\omega \pa{f_1+f_2} + \bar{\boldsymbol\Psi}^- \pa{f_1+f_2}}dxd\omega.
\end{split}
\end{equation}
where $\bar\sigma_0^+=\inf\{\bar\sigma_1+\bar \sigma_2 \}$. We bound the two terms inside the integral on the right-hand side. First thanks to an integration by parts $\eqref{eq:IBP tangential derivative}$.
$$\int_{\Omega\times\S^{d-1}} (f_1-f_2)\nabla_\omega\cdot\cro{\bar{\sigma}^-\nabla_\omega \pa{f_1+f_2}} dx d\omega= -\int_{\Omega\times\S^{d-1}} \bar{\sigma}^- \nabla_\omega(f_1-f_2)\cdot\nabla_\omega(f_1+f_2) dx d\omega.$$
Cauchy-Schwarz with Young's inequalities yields
\begin{equation}\label{dependence 1}
\begin{split}
\abs{\int_{\Omega\times\S^{d-1}} (f_1-f_2)\nabla_\omega\cdot\cro{\bar{\sigma}^-\nabla_\omega \pa{f_1+f_2}}} \leq& \frac{\bar{\sigma}_0^+}{8}\norm{\nabla_\omega(f_1-f_2)}^2_{L^2_{x,\omega}} 
\\&+ \frac{4\norm{\bar{\sigma}^-}^2_{L^\infty_{x,\omega}}}{\bar{\sigma}_0^+}\norm{\nabla_\omega(f_1+f_2)}^2_{L^2_{x,\omega}}.
\end{split}
\end{equation}
The second integrand in \eqref{dependence start} is dealt with in the same way and we get
\begin{equation}\label{dependence 2}
\begin{split}
\abs{\int_{\Omega\times\S^{d-1}} (f_1-f_2)\nabla_\omega\cdot\cro{\bar{\boldsymbol\Psi}^-\pa{f_1+f_2}}} \leq& \frac{\bar{\sigma}_0^+}{8}\norm{\nabla_\omega(f_1-f_2)}^2_{L^2_{x,\omega}} 
\\&+ \frac{4\norm{\bar{\boldsymbol\Psi}^-}^2_{L^\infty_{x,\omega}}}{\bar{\sigma}_0^+}\norm{(f_1+f_2)}^2_{L^2_{x,\omega}}.
\end{split}
\end{equation}
We then plug $\eqref{dependence 1}$ and $\eqref{dependence 2}$ into $\eqref{dependence start}$ and get
\begin{equation*}
\begin{split}
\frac{1}{2}\frac{d}{dt} \norm{f_1-f_2}^2_{L^2_{x,\omega}} \leq& C_2\norm{f_1-f_2}^2_{L^2_{x,\omega}} 
\\&+ \frac{2\norm{\bar{\sigma}^-}^2_{L^\infty_{x,\omega}}}{\bar{\sigma}_0^+}\norm{\nabla_\omega(f_1+f_2)}^2_{L^2_{x,\omega}} +\frac{2\norm{\bar{\boldsymbol\Psi}^-}^2_{L^\infty_{x,\omega}}}{\bar{\sigma}_0^+}\norm{(f_1+f_2)}^2_{L^2_{x,\omega}}
\end{split}
\end{equation*}
on which we apply Gr\"onwall lemma to obtain
\begin{equation}\label{dependence final}
\begin{split}
\norm{f_1-f_2}^2_{L^2_{x,\omega}} \leq & \frac{2}{\bar{\sigma}_0^+}\norm{\bar{\sigma}^-}^2_{L^\infty_{t,x,\omega}}\int_0^t e^{C_2(t-s)}\norm{\nabla_\omega(f_1+f_2)}^2_{L^2_{x,\omega}}\:ds
\\&+\frac{2}{\bar{\sigma}_0^+}\int_0^t e^{C_2(t-s)}\norm{(f_1+f_2)}^2_{L^2_{x,\omega}}\norm{\bar{\boldsymbol\Psi}^-(s)}^2_{L^\infty_{x,\omega}}ds.
\end{split}
\end{equation}
To conclude we apply Proposition \ref{prop:Lp linear} to bound the second term on the right-hand side, which directly gives the expected result.
\end{proof}
\bigskip


\subsubsection{Estimation of vanishing time for velocity averaging densities}\label{subsec:averaging positivity}

Now we turn to an explicit estimation of the vanishing time of velocity averaging densities. 

\bigskip
\begin{prop}\label{prop:positivity}
Let $p$ belong to $[2,+\infty]$ and $f_0 \geq 0$ be in $L^1_{x,\omega} \cap L^p_{x,\omega}$. Let also $\bar{\sigma}$ and $\bar{\boldsymbol\Psi}$ satisfy the assumptions of Theorem \ref{theo:cauchy linear} supplemented with $\bar{\sigma}$ Lipschitz. Let $f$ be the solution to $\eqref{eq:linear Fokker-Planck}$ associated to $f_0$.
For any $\mathbf{K}$ in $W^{1,\infty}_{t,x_*}W^{2,\infty}_{\omega_*}L^\infty_{x,\omega}$ denote
$$\boldsymbol J(t,x,\omega) = \pa{J_i(t,x,\omega)}_{1\leq i \leq d} = \int_{\Omega \times \S^{d-1}} f(t,x_*,\omega_*)\mathbf{K}(t,x,x_*,\omega,\omega_*)dx_*d\omega_*.$$
Then the following holds 
$$\forall (t,x,\omega) \in \R^+\times\Omega\times\S^{d-1},\quad \abs{\boldsymbol J(t,x,\omega)} \geq \abs{ \mathbf{J}(0,x,\omega)}-K_\infty\norm{f_0}_{L^1_{x,\omega}}t$$
where we defined 
\begin{equation}\label{eq:Kinfty}
K_\infty = \sqrt{\sum\limits_{i=1}^dK_{i,\infty}^2}
\end{equation}
with 
\begin{equation}\label{eq:Kiinfty}
K_{i,\infty} = \pa{1+\abs{c}+\norm{\bar{\sigma}}_{L^\infty_{t,x}W^{1,\infty}_\omega}+\norm{\bar{\boldsymbol\Psi}}_{L^\infty_{t,x,\omega}}}\norm{K_i}_{W^{1,\infty}_{t,x_*}W^{2,\infty}_{\omega_*}L^\infty_{x,\omega}}.
\end{equation}
\end{prop}
\bigskip

\begin{proof}[Proof of Proposition \ref{prop:positivity}]
Since $f_0$ is non-negative and belongs to $L^1_{x,\omega} \cap L^p_{x,\omega}$ we deduce from Theorem \ref{theo:cauchy linear} and Proposition \ref{prop:Lp linear} that $f(t)$ is non-negative and belongs to $L^1_{x,\omega} \cap L^p_{x,\omega}$ for all time. Hence, we can multiply $\eqref{eq:linear Fokker-Planck}$ by $\mathbf{K}(x,x_*,\omega,\omega_*)$ and integrate by parts $\eqref{eq:IBP tangential derivative}$. Here again note that there are no additional terms along $\omega_*$ since $\bar{\boldsymbol\Psi}$ and $\nabla_{\omega_*}$ are both orthogonal  to $\omega_*$. This yields
\begin{eqnarray*}
&&\hspace{-0.7cm} \frac{1}{2}\frac{d}{dt}\abs{J_i(t,x,\omega)}^2 =  J_i(t,x,\omega)\int_{\Omega \times \S^{d-1}}K_i(t,x,x_*,\omega,\omega_*)\frac{\partial f}{\partial t}(t,x_*,\omega_*)dx_*d\omega_*
\\&&\hspace{-0.7cm} \quad\quad\quad\quad\quad\quad\quad\quad\quad + J_i(t,x,\omega)\int_{\Omega \times \S^{d-1}}f(t,x_*,\omega_*) \frac{\partial K_i}{\partial t}(t,x,x_*,\omega,\omega_*)dx_*d\omega_*
\\&&\hspace{-0.7cm} \quad=J_i(t,x,\omega)\int_{\Omega \times \S^{d-1}}f\cro{\frac{\partial K_i}{\partial t} + c\omega\cdot\nabla_{x_*}K_i + \nabla_{\omega_*}\cdot\pa{\bar{\sigma} \nabla_{\omega_*}K_i} -\bar{\boldsymbol\Psi}\cdot\nabla_{\omega_*}K_i } dx_* d\omega_*
\\&&\hspace{-0.7cm} \quad\geq -\abs{J_i(t,x,\omega)}\pa{1+\abs{c}+\norm{\bar{\sigma}}_{L^\infty_{t,x}W^{1,\infty}_\omega}+\norm{\bar{\boldsymbol\Psi}}_{L^\infty_{t,x,\omega}}}\norm{K_i}_{W^{1,\infty}_{t,x_*}W^{2,\infty}_{\omega_*}L^\infty_{x,\omega}}\norm{f(t)}_{L^1_{x,\omega}}
\\&&\hspace{-0.7cm} \quad\geq - K_{i,\infty} \norm{f_0}_{L^1_{x,\omega}}\abs{J_i(t,x,\omega)}
\end{eqnarray*}
where we used the conservation of mass along time. Summing over $i$ we obtain
\begin{eqnarray*}
\frac{1}{2}\frac{d}{dt}\abs{\mathbf{J}(t,x,\omega)}^2 &\geq& -\norm{f_0}_{L^1_{x,\omega}}\sum\limits_{i=1}^dK_{i,\infty}\abs{J_i(t,x,\omega)}
\\&\geq & -\norm{f_0}_{L^1_{x,\omega}}\sqrt{\sum\limits_{i=1}^dK_{i,\infty}^2}\abs{\mathbf{J}}
\end{eqnarray*}
 The latter implies that
$$\abs{\mathbf{J}(t,x,\omega)} \geq \abs{\mathbf{J}(0,x,\omega)} -\sqrt{\sum\limits_{i=1}^dK_{i,\infty}^2}\norm{f_0}_{L^1_{x,\omega}}t$$
which concludes the proof.
\end{proof}
\bigskip

\subsection{The local-in-time nonlinear Cauchy theory}\label{sec:local cauchy}

The present section is devoted to the proof of Theorem \ref{theo:local cauchy} thanks to the linear study we presented in Sec. \ref{sec:linear fokker planck}. We shall prove existence, uniqueness and global existence criterion to
\begin{equation}\label{eq:local cauchy NL}
\partial_t f + c \omega\cdot\nabla_x f = \nabla_\omega \cdot\pa{\sigma(f) \nabla_\omega f} + \nabla_\omega\cdot\pa{\nu(f)f\mathbf{P}_{\omega^\bot}\cro{\boldsymbol\Psi[f]}}.
\end{equation} 

\bigskip
\begin{proof}[Proof of Theorem \ref{theo:local cauchy}]
We fix $p$ in $[2,+\infty]$ and $f_0 \geq 0$ in $L^1_{x,\omega}\cap L^p_{x,\omega}$. We recall some assumptions of Theorem \ref{theo:local cauchy}:
\begin{eqnarray*}
\int_{\Omega\times\S^{d-1}}f_0(x,\omega)\:dxd\omega &=& M_0 >0
\\ \inf\limits_{(x,\omega) \in \Omega\times\S^{d-1}} \abs{\boldsymbol J[f_0](x,\omega)}_\alpha &=& J_0 > 0.
\end{eqnarray*}
The strategy is to apply a contraction fix point argument so we start by defining  a complete metric space on which we shall work. For any $M\geq M_0$ and $T>0$ we define
\begin{equation}\label{GammaT}
\Gamma^M_T = \br{f \in L^2_{[0,T],x,\omega}\:,\quad\mbox{and} \begin{array}{l}\esssup\limits_{[0,T]}\norm{f(t)}_{L^1_{x,\omega}} \leq M  \\ \esssup\limits_{[0,T]\times\Omega\times\S^{d-1}}\frac{\abs{\nu(f(t))}}{\abs{\boldsymbol J[f(t)](x,\omega)}_\alpha}\leq \frac{\nu_\infty}{\alpha+(1-\alpha)\frac{J_0}{2}}\end{array}}.
\end{equation}
Note that since $J[f]$ is an integral over $(x_*,\omega_*)$ against $\mathbf{K}$ and because $\mathbf{K}$ is $L^2_{x_*,\omega_*}$ (by $(H1)$), a Cauchy sequence for the $L^2$-norm in $\Gamma^M_T$ indeed converges in $\Gamma^M_T$.

For a function $g$ in $\Gamma^M_T$ we have that $g$ belongs to $L^2_{[0,T],x,\omega}$ so by Fubini's theorem $g$ belongs to $L^2([0,T],L^2_{x,\omega})$ and for almost every $t$ in $[0,T]$ the function $g(t,\cdot,\cdot)$ belongs to $L^2_{x,\omega}$. Therefore $\sigma(g(t))$ and $\nu(g(t))$ are defined almost everywhere in $[0,T]$ and with hypotheses $(H2)$ and $(H3)$ one has
$$\sigma_0 \leq \sigma(g(t)) \leq \sigma_\infty \quad\mbox{and}\quad \norm{\nu(g(t))}_{L^\infty_xW^{1,\infty}_\omega} \leq \nu_\infty \quad\mbox{almost everywhere in }[0,T].$$
Defining
$$\bar{\boldsymbol\Psi}(t,x,\omega) = \nu(g(t))(x,\omega)\mathbf{P_{\omega^\bot}}\cro{\boldsymbol\Psi[g]}$$
hypotheses $(H1)$ shows for $t\in[0,T]$
\begin{eqnarray}
\abs{\bar{\boldsymbol\Psi}(t,x,\omega)} &\leq&  \nu_\infty\frac{\abs{\int_{\Omega\times\S^{d-1}}g(t,x_*,\omega_*)\mathbf{K}(t,x,x_*,\omega,\omega_*)dx_*d\omega_*}}{\abs{\boldsymbol J[g](t,x,\omega)}_\alpha} \nonumber
\\&\leq& \nu_\infty\left\{\begin{array}{ll}  \norm{\mathbf{K}}_{L^\infty_{t,x,x_*,\omega,\omega_*}}M &\quad\mbox{if}\quad \alpha =1\\ \frac{1}{1-\alpha}, &\quad\mbox{if}\quad \alpha \neq 1 \end{array}\right. := C^1_M \label{C1M}
\end{eqnarray}
and
\begin{eqnarray}
\abs{\nabla_\omega\cdot \bar{\boldsymbol\Psi}(t,x,\omega)} &\leq& \abs{\nabla_\omega\nu(g) \cdot \mathbf{P_{\omega^\bot}}\cro{\boldsymbol\Psi[g]}} + \abs{\nu(g)\nabla_\omega\cdot \mathbf{P_{\omega^\bot}}\cro{\boldsymbol\Psi[g]}}\nonumber
\\&\leq& C^1_M + \nu_\infty C_\alpha M\norm{\mathbf{K}}_{L^\infty_{t,x,x_*,\omega_*}W^{1,\infty}_{\omega}}:= C^2_M\label{C2M}
\end{eqnarray}
where $C_\alpha >0$ only depends on $\alpha$ and equals $1$ when $\alpha=1$. Therefore thanks to Theorem \ref{theo:cauchy linear} we are able to define
$$\function{\Phi}{\Gamma^M_T}{L^2_{[0,T],x,\omega}}{g}{\Phi(g) = f_g},$$
where $f_g$ is the solution on $[0,T]\times\Omega\times\S^{d-1}$ to the linear equation
$$\partial_t f + c \omega\cdot\nabla_x f = \nabla_\omega\cdot\pa{\sigma(g)(t,x,\omega)\nabla_\omega f + f \nu(g)\mathbf{P_{\omega^\bot}}\cro{\boldsymbol\Psi[g]}}$$
associated to the initial datum $f_0$.

\bigskip
We now show that for a specific $T$, $\Phi$ is in fact a contraction from $\Gamma^M_T$ to $\Gamma^M_T$. Due to Proposition \ref{prop:Lp linear} and since $f_0$ belongs to $L^2_{x,\omega}$ we see that $\Phi(g)$ belongs to $L^\infty([0,T])L^2_{x,\omega}) \subset L^2_{[0,T]\times\Omega\times\S^{d-1}}$ and also that for almost every $t\in[0,T]$
$$\norm{\Phi(g)(t)}_{L^1_{x,\omega}}=M_0 < M.$$
Moreover, Proposition \ref{prop:positivity} gives that almost everywhere
\begin{equation}\label{eq:positive renormalisation}
\abs{\boldsymbol J[\Phi(g)](t,x,\omega)} \geq J_0 -K_\infty M_0 T. 
\end{equation}
We defined $K_\infty$ in $\eqref{eq:Kinfty}$ and we recall it
\begin{equation}\label{eq: Kinfty}
K_\infty(f_0) = \pa{1 + \abs{c}+ \sigma_\infty+\sigma_{\tiny{lip}} + C^1_M}\sqrt{\sum\limits_{n=1}^{d}\norm{K_i}^2_{W^{1,\infty}_{t,x_*}W^{2,\infty}_{\omega_*}L^\infty_{x,\omega}}}.
\end{equation}
Therefore if we choose
\begin{equation}\label{T0}
T < T_0 := \frac{J_0}{2K_\infty M_0},
\end{equation}
the lower bound $\eqref{eq:positive renormalisation}$ implies
$$\forall t \in [0,T],\quad \abs{\boldsymbol J[\Phi(g)](t,x,\omega)}_\alpha \geq \alpha+(1-\alpha)\frac{J_0}{2},$$
so that
$$\forall t \in [0,T],\quad \frac{\abs{\nu(\Phi(g))}}{\abs{\boldsymbol J[\Phi(g)](t,x,\omega)}_\alpha} \leq \frac{\nu_\infty}{\alpha+(1-\alpha)\frac{J_0}{2}}$$
and we thus proved that for any $0<T<T_0$, $\Phi$ maps $\Gamma_T^M$ to itself.

\bigskip
It remains to prove that $\Phi$ is a contraction on $\Gamma_T^M$ for $T$ sufficiently small. By hypothesis $(H\ref{Hsigma})$ we have
$$\norm{\sigma(g_1) - \sigma(g_2)}_{L^\infty_{[0,T],x,\omega}} \leq \sigma_{\mbox{\tiny{lip}}}\norm{g_1-g_2}_{L^2_{[0,T],x,\omega}}.$$
Also, thanks to the Lipschitz property $(H\ref{Hnu})$ of the friction $\nu$ and the kernel form of $\boldsymbol\Psi[g]$ one infers
$$\int_0^t\norm{\bar{\boldsymbol\Psi}(g_1) - \bar{\boldsymbol\Psi}(g_2)}^2_{L^\infty_{x,\omega}} \leq \Psi_{\mbox{\tiny{lip}}}\norm{g_1-g_2}^2_{L^2_{[0,t],x,\omega}}$$
where $\Psi_{\mbox{\tiny{lip}}} >0$ is a constant only depending on $\norm{\mathbf{K}}_{L^\infty_{t,x,x_*,\omega,\omega_*}}$, $M$, $\alpha$, $\Psi_0$ and $\nu_{\mbox{\tiny{lip}}}$. This inequality comes from the algebraic identity $ab-cd = \frac{(a-c)(b+d)}{2} + \frac{(a+c)(b-d)}{2}$. 
\par We apply Proposition \ref{prop:dependence coef} to see that for any $g_1$ and $g_2$ in $\Gamma_T^M$
$$\norm{\Phi(g_1)-\Phi(g_2)}^2_{L^2_{[0,T],x,\omega}} \leq \Lambda(T) \norm{g_1-g_2}^2_{L^2_{[0,T],x,\omega}}$$
where
\begin{equation}\label{Lambda(T)}
\Lambda(T) = \frac{2 \norm{f_0}^2_{L^2_{x,\omega}}}{\sigma^2_0}Te^{3C_2T}\pa{\Psi_{\mbox{\tiny{lip}}}\sigma_0 + 2\sigma_{\mbox{\tiny{lip}}}}.   
\end{equation}
Note that we used Proposition \ref{prop:Lp linear} to obtain explicit constants. To conclude it suffices to choose
\begin{equation}\label{T1}
T < \min\br{T_0,\: \sup\limits_{t>0}\Lambda(t) \leq \frac{1}{4}}:=T_1
\end{equation}
because then for any $g_1$ and $g_2$ in $\Gamma^M_T$ we have proved
$$\norm{\Phi(g_1)-\Phi(g_2)}_{L^2_{[0,T],x,\omega}} \leq \frac{1}{2} \norm{g_1-g_2}_{L^2_{[0,T],x,\omega}},$$
which implies that $\Phi$ is a contraction on $\Gamma^M_T$.

\bigskip
We thus proved the local existence and uniqueness of a solution $f$ to the nonlinear kinetic equation $\eqref{eq:general kinetic orthogonal}$ in $L^2_{[0,T_{\max}),x,\omega}$ with $T_{\max} \geq T_1$. The solution $f$ belongs to $L^\infty\pa{[0,T_{\max}),L^p_{x,\omega}} \cap L^2\pa{[0,T_{\max}),L^2_xH^1_\omega}$ due to Proposition \ref{prop:Lp linear} since $\sigma(f)$ and $\nu(f)$ are well-defined and satisfy the required assumptions, as we saw above. At last, thanks to the global in time Cauchy theory for the linear equation given by Theorem \ref{theo:cauchy linear} we see that if $\lim\limits_{t\to T_{\max}^-}\sup\limits_{\Omega \times\S^{d-1}}\frac{\nu(f)}{\abs{\boldsymbol J[f]}_\alpha}<+\infty$, then we can apply our fixed point argument starting at $T_{\max}$ and therefore $T_{\max}$ must be $+\infty$. Which concludes the proof of Theorem \ref{theo:local cauchy}.
\end{proof}
\bigskip

\subsection{Global existence and decay of the free energy}\label{sec:long time behaviour}

The present section focuses on some important cases where one is ensured to have globally defined solution to the fully non-linear kinetic equation $\eqref{eq:general kinetic orthogonal}$. Then we get a look into the free energy and its energy dissipation.

\bigskip

\subsubsection{Important examples of global solutions}\label{subsec:global examples}

Theorem \ref{theo:local cauchy} offers a direct and easy way to check if the solutions obtained are globally defined. Indeed, under the assumptions $(H1)-(H2)-(H3)$, if one shows that $\inf\limits_{\Omega\times\S^{d-1}}\abs{\boldsymbol J[f]}_\alpha$ cannot vanish along the flow then indeed $\sup\limits_{\Omega\times\S^{d-1}}\frac{\abs{\nu(f)}}{\abs{\boldsymbol J[f]}_\alpha}$ cannot explode in finite time.

\bigskip
\textbf{The non-normalised case: $\mathbf{\boldsymbol\alpha \neq 0}$.} The discussion above implies that when $\alpha \neq 0$ then $\abs{\boldsymbol J[f]}_\alpha \geq \frac{1}{\alpha}>0$ for any function $f$ so that $\boldsymbol\Psi[f]$ is well-defined. Therefore for any $\alpha \neq 0$ the solutions constructed in Theorem \ref{theo:local cauchy} are global in time. This is the case of the equations looked at in \cite{Ons,Doi,FrouLiu} among others.

\bigskip
\textbf{The non-asymptotic normalised case: $\mathbf{\frac{\boldsymbol\nu(f)}{\abs{\boldsymbol J[f]}}}$ prevents any explosion.} Here we are interested in the normalised framework $\alpha =0$ but when the coefficient $\nu(f)$ compensate a possible vanishing of $\abs{\boldsymbol J[f]}$. Namely if we add the assumption that $f\mapsto\frac{\nu(f)}{\abs{\boldsymbol J[f]}}$ is a bounded function when $\abs{\boldsymbol J[f]}$ tends to zero. Again, in this particular case the solutions constructed in Theorem \ref{theo:local cauchy} are global in time. This is the case of equations looked at in \cite{DegFrouLiu1,DegFrouLiu2} where $\nu(f) = \nu(\abs{\mathbf{J}_f})$ with $\frac{\nu(y)}{y}$ bounded and is generalised here for instance to any $\nu(f) = \nu(\abs{\boldsymbol J[f]})$ with $\frac{\nu(y)}{y}$ bounded near zero.

\bigskip
\textbf{The normalised case of kernel with one coordinate with a strict sign.} In this paragraph we assume that the kernel $\mathbf{K}(t,x,x_*,\omega,\omega_*)$ is such that one of its coordinate $K_i$ has a strict sign in the following sense: there exists a positive constant $k_i>0$ such that
$$\inf\limits_{\R^+\times\Omega\times\Omega\times\S^{d-1}\times\S^{d-1}}K_i \geq k_i \quad\mbox{or}\quad\sup\limits_{\R^+\times\Omega\times\Omega\times\S^{d-1}\times\S^{d-1}} K_i \leq -k_i.$$
It is a direct verification since, thanks to the positivity of $f$, either
$$\int_{\Omega\times\S^{d-1}}fK_i\:dx_*d\omega_* \geq k_i M_0 \quad\mbox{or}\quad \int_{\Omega\times\S^{d-1}}f K_i\:dx_*d\omega_* \leq -k_i M_0.$$
Therefore, $J_i(t,x,\omega)$ cannot vanish in finite time and again, in this particular case the solutions constructed in Theorem \ref{theo:local cauchy} are global in time.

\bigskip
\textbf{The Vicsek kernel operator.}  We study the special case when $\mathbf{K}(t,x,x_*,\omega,\omega_*) = \omega_*$ and normalised $\alpha=0$, as in \cite{DegFrouLiu1,DegFrouLiu2,GamKan,FigKanMor} for the spatially homogeneous equation. Note that here we also obtain the global in time for the non-spatially homogeneous setting. We also assume that $\nu$ and $\sigma$ do not depend on $x$ nor $\omega$ and that the friction $\nu$ is negative (standard in all previous studies).
\par Contrary to the setting above, the present kernel can degenerate in the sense that one non-zero coordinate of $\mathbf{K}$ can possibly vanish later on but giving rise to another non-zero coordinate. As a consequence, the full $\abs{\mathbf{K}}$ will not vanish in finite time. First we explicitly write
$$J_i(t,x,\omega) = J_i(t) =\int_{\Omega\times\S^{d-1}}\omega_{i*}f(t,x_*,\omega_*)\, dx_* d\omega_* \quad\mbox{and}\quad \boldsymbol\Psi = \frac{\boldsymbol J}{\abs{\boldsymbol J}}.$$
As before let us consider the computation below for $0 \leq t < T_{\max}$ and recall $\eqref{eq:global equality}$ which in the present setting yields
\begin{equation*}
\frac{1}{2}\frac{d}{dt}\abs{\boldsymbol J}^2 = \sum\limits_{i=1}^d J_i(t) \int_{\Omega \times \S^{d-1}}f\cro{\sigma(f)(t)\Delta_{\omega_*}(\omega_{*i}) -  \nu(f)(t)\mathbf{P}_{{\omega_*^\bot}}[\boldsymbol\Psi]\cdot\nabla_{\omega_*}(\omega_{*i})}.
\end{equation*}
Thanks to the properties of the tangential derivatives we can compute directly with $\eqref{eq:tangential derivative omega}$
\begin{eqnarray*}
\frac{1}{2}\frac{d}{dt}\abs{\boldsymbol J(t)}^2 &=& -(d-1)\sigma(f)(t)\abs{\boldsymbol J(t)}^2 - \nu(f)(t)\int_{\Omega\times\S^{d-1}}\frac{1}{\abs{\boldsymbol J(t)}}\mathbf{P}_{\omega_*^\bot}[\boldsymbol J(t)]\cdot \nabla_{\omega_{*}}\pa{\omega_{*}\cdot\boldsymbol J(t)}
\\&=& -(d-1)\sigma(f)(t)\abs{\boldsymbol J(t)}^2 -\frac{\nu(f)(t)}{\abs{\boldsymbol J(t)}}\int_{\Omega\times\S^{d-1}}\abs{\nabla_{\omega_*}\pa{\omega_*\cdot\boldsymbol J(t)}}^2
\\&\geq&-(d-1)\sigma_0\abs{\boldsymbol J(t)}^2.
\end{eqnarray*}
Note that we use the fact that the friction is negative. Using Gr\"onwall inequality we conclude that $\abs{\boldsymbol J(t)}^2 \geq \abs{\boldsymbol J(0)}^2e^{-2(d-1)\sigma_0t}$.
\par These computations were already done in an homogeneous regularised setting in \cite{FigKanMor} and our local theory allows us to carry them out directly. Such a lower bound thus imply the non-vanishing of $\abs{\boldsymbol J(t)}$ in finite time and thus leading again to global in time solution in this setting.

\bigskip

\subsubsection{Decay of the free energy}\label{subsec:convergence to equilibrium}

We recall our definitions for the free energy and the energy dissipation $\eqref{eq:free energy}-\eqref{eq:energy dissipation}$
\begin{eqnarray*}
\mathcal{F}[f](t) &=& \int_{\Omega \times \S^{d-1}} f\mbox{ln}(f)dxd\omega \nonumber
\\& &+ \int_0^t \int_{\Omega \times \S^{d-1}} f \cro{\nabla_\omega\cdot\pa{\nu(f)\mathbf{P}_{\omega^\bot}[\boldsymbol\Psi[f]]} - \frac{\nu(f)^2}{\sigma(f)}\abs{\mathbf{P}_{\omega^\bot}[\boldsymbol\Psi[f]]}^2}dxd\omega ds
\\\mathcal{D}[f](t) &=& \int_{\Omega\times\S^{d-1}} \sigma(f) f \abs{\nabla_\omega \mbox{ln}(f) + \frac{\nu(f)}{\sigma(f)}\mathbf{P}_{\omega^\bot}[\boldsymbol\Psi[f]]}^2dxd\omega.
\end{eqnarray*}
Let $f_0$, $\sigma$, $\nu$ and $\mathbf{K}$ be as in Theorem \ref{theo:local cauchy}. We now establish the decay of $\mathcal{F}[f](t)$:
\begin{equation}\label{eq:decay F[f]}
\forall t \geq 0,\quad \frac{d}{dt}\mathcal{F}[f](t) = -\mathcal{D}[f](t).
\end{equation}
 It comes from direct computations. Indeed since $\nabla_{x/\omega}\cro{1+\mbox{ln}f} = \frac{\nabla_{x/\omega} f}{f}$ it follows by the integrating by parts formula $\eqref{eq:IBP tangential derivative}$,
\begin{eqnarray*}
\frac{d}{dt}\int_{\Omega \times \S^{d-1}} f\mbox{ln}f(t,x,\omega)dxd\omega &=& \int_{\Omega \times \R^d}\cro{1 + \mbox{ln}f}\partial_t f(t,x,\omega)dxd\omega
\\&=& \int_{\Omega \times \S^{d-1}}c\omega\cdot \nabla_x f\:dxd\omega 
\\&&+ \int_{\Omega \times \S^{d-1}}\cro{\sigma(f)\nabla_\omega f + \nu(f) f \mathbf{P}_{\omega^\bot}[\boldsymbol\Psi[f]]}\cdot \frac{\nabla_{\omega} f}{f}dxd\omega
\\&=& - \int_{\Omega \times \S^{d-1}}\cro{\sigma(f)\frac{\abs{\nabla_{\omega} f}^2}{f}+\nu(f)\mathbf{P}_{\omega^\bot}[\boldsymbol\Psi[f]]\cdot\nabla_\omega f}dxd\omega.
\end{eqnarray*}
Hence,
\begin{eqnarray*}
\frac{d}{dt}\mathcal{F}[f](t)&=& - \int_{\Omega\times\S^{d-1}}\cro{\sigma(f)\frac{\abs{\nabla_{\omega} f}^2}{f}+ 2\nu(f)\mathbf{P}_{\omega^\bot}[\boldsymbol\Psi[f]]\cdot\nabla_\omega f +\frac{\nu(f)^2}{\sigma(f)}\abs{\mathbf{P}_{\omega^\bot}[\boldsymbol\Psi[f]]}^2 f }dxd\omega
\\&=& -\mathcal{D}[f](t)
\end{eqnarray*}
as claimed above.

\bigskip

\section{Mean-field limit: proofs of the results in Section \ref{subsec:particles intro}}\label{sec:particles}

Let us first describe the strategy we are about to implement. At the core, we will use the coupling approach popularised by Sznitmann  \cite{sznitman1991topics} (which differs from classical BBGKY approaches \cite{jabin2017mean}). We will prove theorems \ref{th:existence_solutions} and \ref{th:convergence_process}. We will follow the same steps  in \cite{BolCanCar2} and the results in \cite{carmona2016lectures}. 
	
\bigskip

\subsection{Preliminaries: functional framework and notations.}
\label{sec:preliminaries_SDE}

In this section we will work in the space $\mathcal{H}$ defined in Section \ref{sec:functional_framework} and with the 2-Wasserstein distance also defined there. We will need some results on Wasserstein distances in the sequel, that we summarise next (these results and proofs can be found in \cite{villani2008optimal}). In our setting, the Wasserstein distance of order 1 is given by
\beqar
&&W_1(m,p) := \inf \Big\{ \left[\int_{(\Omega\times\R^{d})^2}|z-u|\, \pi(dz, du)\right];\\
&& \qquad\qquad\qquad\qquad\, \pi \in \mathcal{P}(\Omega\times\R^{d})^2\mbox{ with marginals $m$ and $p$}\Big\}.
\eeqar
This distance can be expressed in duality form using the Kantorovich-Rubinstein distance
\be \label{eq:Kantorovich_distance} 
W_1(m,p) = \sup_{\|\varphi\|_{\tiny{lip}}\leq 1} \left\{ \int_{\Omega\times\R^{d}} \varphi(z) dm -\int_{\Omega\times\R^{d}} \varphi(z) dp\right\}.
\ee
Also, by H\"older's inequality, it holds
\be \label{eq:order_W}
W_1\leq W_2.
\ee

\subsection{The non-singular dynamics.}
\label{sec:non-singular-dynamics}
To  prove Theorem \ref{th:existence_solutions} and \ref{th:convergence_process}, we will consider first modified versions of Systems \eqref{eq:particle_dyn} and \eqref{eq:auxiliary} where we work with Lipschitz coefficients. 
Being more precise, firstly we construct $\func{\gamma}{\R^d}{\R^d}$ Lipschitz and bounded with
\be 
\gamma(v) = v, \qquad\mbox{when } |v|\leq 2.
\ee
We will use the function $\gamma$ as a substitute of the variable $V(t)$ in some parts of the equations. We do this to be able to apply known results of existence an uniqueness of solutions for stochastic differential equations that  require the coefficients to be Lipschitz and bounded (see Th. \ref{th:existence_particles_carmona}). We will prove in a second stage that along the dynamics it holds that $|V(t)| =1$, therefore $\gamma(V(t))=V(t)$ and we recover the terms in the original equation.  Secondly, we define a functional $\tau_0$ as
\be \label{eq:def_tau0}
\tau_0(x,v,m) = \tau_{\eps_0}(x,\gamma(v),m).
\ee

We will show that $\tau_0$ is Lipschitz and bounded in the whole $\mathcal{H}$:
\begin{lemma}
\label{lem:lipschitz_property}
The functions $r=r(x,v,m) :=|\mathbf{J}|(x,\gamma(v),m):\mathcal{H}\to \R_+$ and $\tau_0=\tau_0(x,v,m):\mathcal{H}\to \R_+$ are Lipschitz and bounded. 
\end{lemma}

\begin{proof}[Proof of Lemma \ref{lem:lipschitz_property}]
We notice, first, that since $\gamma$ is bounded, then $K(x, \gamma(v),m)$ is bounded in all the variables and, therefore, $r_f(z)$ is bounded in $\mathcal{H}.$ Now, let $f,g\in \mathcal{P}_2(\Omega\times \R^d)$ and fix $z\in \Omega\times \R^d$,  we show next that $r_f:=r(z,f)$ is Lipschitz in $f$. Consider $z,z'\in \Omega\times \R^d$ and for $z=(x,v)$, denote $\tilde{z}=(x, \gamma(v))$. Now, we consider (the integrals are in $\Omega\times \R^d$)
\beqarl
| r_f(z)-r_g(z')|
 &= &\left|\, |\int K(\tilde z,y)f(dy)|- |\int K(\tilde z',y)\, g(dy)| \,\right| \nonumber\\
&\leq& \Big|\int[K(\tilde z,y) f(dy)-K(\tilde z',y) g(dy)]\Big| \nonumber\\
&\leq& \|K\|_{Lip} |\tilde z - \tilde z'|+\left| \int K(\tilde z',y) [f(dy)-g(dy)]\right|\nonumber\\
 &\leq & \|K\|_{Lip} (\|\gamma\|_{Lip}+1) |z - z'|\nonumber\\
 && \qquad \qquad+\|K\|_{Lip}\,  \left| \sup_{\|\varphi\|_{Lip}\leq 1} \lp  \int \varphi(y) f(dy) - \int \varphi(y) g(dy) \rp \right|  \nonumber\\
 &=& \|K\|_{Lip} (\|\gamma\|_{Lip}+1)|z - z'|+ \|K\|_{Lip} W_1(f,g)  \nonumber\\
 &\leq & \|K\|_{Lip} (\|\gamma\|_{Lip}+1)|z - z'|+ \|K\|_{Lip} W_2(f,g) \nonumber
\eeqarl
where in the second line we used the reverse triangle inequality; in the third line we have used that $\int_{\R^{2d}} f =1$ and that $K$ is Lipschitz; in the fourth line we used that $\gamma$ is Lipschitz and choose an arbitrary $z_0\in \R^{2d}$; the fifth line is given by  \eqref{eq:Kantorovich_distance}; and the last inequality follows from \eqref{eq:order_W}. Therefore, we conclude that $r=r_f(z)$ is Lipschitz in $\mathcal{ H}$.

\medskip
Next, we also show that $\tau_0$ is Lipschitz and bounded in $\mathcal{H}$. First, whenever $r\geq \eps_0$ we have that $(\alpha + (1-\alpha)r)^{-1}$ is also Lipschitz and bounded for all $\alpha \in [0,1]$, therefore, $\tau_0$ is product of Lipschitz and bounded functions, hence it is Lipschitz and bounded.
\end{proof}

With the function $\tau_0$ we define the \textit{non-singular particle dynamics} as:
\begin{subequations} \label{eq:non-singular-particle_dyn}
\begin{numcases}{} 
dX^{i,N}_t = V^{i,N}_t dt,\\
dV_{t}^{i,N} = \nu(\mu^N)\, \mathbf{P}_{(V^{i,N})^\perp} (\tau_0(X^{i,N},V^{i,N},\mu^N)) dt \nonumber\\
\qquad\quad+\frac{1}{2} \mathbf{P}_{(\boin)^\perp}[(\nabla_v\sigma(\mu^N))(\bxin,\boin)]\, dt\nonumber\\
\qquad\quad + \sqrt{2\sigma(\mu^N)}\, \mathbf{P}_{(V^{i,N})^\perp} \circ dB^i_t,
\end{numcases}
\end{subequations}
(this is exactly System \eqref{eq:particle_dyn} substituting $\tau_0$ by the function $\tau_{\eps_0}$)
and the \textit{non-singular auxiliary process} is given by:
\begin{subequations}\label{eq:non-singular-auxiliary}
\begin{numcases}{} 
d\bxin_t = \boin_t dt,\\
d\boin = \nu(f)\, \mathbf{P}_{(\boin)^\perp} (\tau_0(\bxin,\boin,f)) dt \nonumber\\
\qquad\quad+\frac{1}{2} \mathbf{P}_{(\boin)^\perp}[(\nabla_v\sigma(f))(\bxin,\boin)]\, dt\nonumber\\
\qquad\quad+ \sqrt{2\sigma(f)}\, \mathbf{P}_{(\boin)^\perp} \circ dB^i_t,\\
f_t = \mbox{law}(\bxin, \boin),
\end{numcases}
\end{subequations}
(again this is exactly System \eqref{eq:auxiliary}  after substituting $\tau_0$ by $\tau_{\eps_0}$).
Analogously, we also consider the \textit{non-singular kinetic equation} in $\Omega \times \mathbb{S}^{d-1}$ given by
\be \label{eq:non-singular-kinetic}
\partial_t f + \omega\cdot \nabla_x f = \nabla_\omega \cdot (\sigma(f) \nabla_\omega f) + \nabla_\omega \cdot(\nu(f) f \nabla_\omega\,\tau_0(f)).
\ee

Note that under the assumptions of Theorem \ref{theo:local cauchy}, \eqref{eq:non-singular-kinetic}  has global existence and uniqueness of solutions in the spaces stated in Theorem \ref{theo:local cauchy}.

\begin{remark}
\label{rem:lipschitz_condition}
Notice that we are assuming that $\nu$, $\sigma$ and $\nabla_v \sigma$ are bounded and Lipschitz for all $v\in \R^d$ rather than $v\in \mathbb{S}^{d-1}$ (see hypothesis (H\ref{Hmeanfield})). However, we just need that $\nu$, $\sigma$ and $\nabla_v \sigma$ to be Lipschitz and bounded in a neighbourhood of $|v|=1$, as we can use regularising arguments as the one done to define the functional $\tau_0$.
\end{remark}

\begin{prop}
\label{prop:results-non-singular}
Theorems \ref{th:existence_solutions} and \ref{th:convergence_process} hold replacing in the statements System \eqref{eq:particle_dyn} by System \eqref{eq:non-singular-particle_dyn}; System \eqref{eq:auxiliary} by System \eqref{eq:non-singular-auxiliary}; and the kinetic equation \eqref{eq:general kinetic orthogonal} by the kinetic equation \eqref{eq:non-singular-kinetic}. 
\end{prop}

We prove this Proposition in Section \ref{sec:prop_non_singular}. The proof of Th. \ref{th:existence_solutions} and \ref{th:convergence_process} is direct assuming Prop. \ref{prop:results-non-singular} to hold true:

\begin{proof}[Proof Th. \ref{th:existence_solutions} and \ref{th:convergence_process} assuming Prop. \ref{prop:results-non-singular}]
The result is direct by Lem. \ref{lem:norm_1_particles}, since it implies that the non-singular dynamics and the approximated dynamics coincide.
\end{proof}

\subsection{Proof of Theorem \ref{th:NormalizedCase}}
\label{sec:proof_mean-field-theorems}

%

Let $f_t$ be solution to the Vicsek kinetic equation \eqref{eq:general kinetic orthogonal}. By the proof of Theorem \ref{theo:local cauchy} we know that for $T<T_1$ (with $T_1$ given in Eq. \eqref{T1}) it holds that 
\be \label{eq:estimate-1}
\forall (t,x,\omega) \in [0,T]\times\Omega\times\S^{d-1},\quad \abs{\boldsymbol J(t,x,\omega,f_t)} > c_*.
\ee
Therefore, for $t\leq T$ the approximated kinetic equation \eqref{eq:general kinetic approx} coincides with the Vicsek kinetic equation \eqref{eq:general kinetic orthogonal}. \\
Suppose now that for all $a>0$ it holds that
\be \label{eq:first_inequality}
\mathbb{P}\lp \sup_{t\in[0,T]}\sup_{(x,v)} \left||\mathbf J(t,x,v,\mu^N_t)|-|\mathbf J(t,x,v,f_t)|\right|\geq a\rp \leq \frac{1}{a}\eps(N),
\ee
where $\eps(N)\to 0$ as $N\to \infty$.
If \eqref{eq:first_inequality} holds, then we have that
\beqarl
&&\mathbb{P}\lp \inf_{t\in[0,T]} \inf_{(x,v)} |\mathbf J(t,x,v,\mu^N_t)|> \eps_0\rp = 1- \mathbb{P}\lp \inf_{t\in[0,T]}\inf_{(x,v)} |\mathbf J(t,x,v,\mu^N_t)|\leq\eps_0\rp  \nonumber\\
 &&\geq   1-\mathbb{P}\lp\inf_{t\in[0,T]}\inf_{(x,v)}\left| |\mathbf J(t,x,v,f_t)-\mathbf J(t,x,v,\mu^N_t)|- |\mathbf J(t,x,v,f_t)|\right|\leq\eps_0\rp, \label{eq:boundprob1}
\eeqarl
where we used the following inequality:
$$\inf_{t\in[0,T]}\inf_{(x,v)}|\mathbf J(t,x,v,\mu^N_t)|\geq \inf_{t\in[0,T]}\inf_{(x,v)}\left||\mathbf J(t,x,v,\mu^N_t)- \mathbf J(t,x,v, f_t)|- |\mathbf J(t,x,v,f_t) |\right|,$$
which follows from the triangular inequality  ($||a|-|b|| \leq |a-b|$ for all $a,b\in \R$).\\
The following bound holds
\beqar
&&\inf_{t\in[0,T]}\inf_{(x,v)}\left| |\mathbf J(t,x,v,f_t)-\mathbf J(t,x,v,\mu^N_t)|- |\mathbf J(t,x,v,f)|\right| \\
\qquad&&\geq \inf_{t\in[0,T]}\inf_{(x,v)} |\mathbf J(t,x,v,f_t)|- |\mathbf J(t,x,v,f_t)-\mathbf J(t,x,v,\mu^N_t)|\\
\qquad&&\geq \inf_{t\in[0,T]} \inf_{(x,v)} |\mathbf J(t,x,v,f_t)| - \sup_{(x,v)}|\mathbf J(t,x,v,f_t)-\mathbf J(t,x,v,\mu^N_t)|\\
\qquad &&\geq c_*- \sup_{t\in[0,T]}\sup_{(x,v)}|\mathbf J(t,x,v,f_t)-\mathbf J(t,x,v,\mu^N_t)|,
\eeqar
by \eqref{eq:estimate-1}.
Using this inequality in combination with \eqref{eq:boundprob1} we deduce that
$$
\mathbb{P}\lp \inf_{t\in[0,T]}\inf_{(x,v)} |\mathbf J(t,x,v,\mu^N_t)|> \eps_0\rp \geq 1 - \mathbb{P} \lp \sup_{t\in[0,T]}\sup_{(x,v)}\left|\mathbf{J}(t,x,v,\mu^N_t) - \mathbf J(t,x,v, f)\right| \geq c_*-\eps_0 \rp.
$$
Finally, applying \eqref{eq:first_inequality} in this inequality we obtain \eqref{eq:highprob}, from which we conclude the proof of the theorem. \\
 We are left with checking that \eqref{eq:first_inequality} holds. We have the following applying Markov's inequality:
\beqarl \label{eq:estimate0}
&&\mathbb{P} \lp \sup_{t\in[0,T]}\sup_{(x,v)}\left|\mathbf J(t,x,v,\mu^N_t)-\mathbf J(t,x,v,f_t)\right|\geq a\rp \nonumber\\
&& \leq  \frac{\mathbb{E}\left[\sup_{t\in[0,T]}\sup_{(x,v)}\left|\mathbf J(t,x,v,\mu^N_t)-\mathbf J(t,x,v,f_t)\right|\right]}{a}. \label{eq:estimate0}
\eeqarl
Now, let $C(\mathbf{K})$ be the Lipschitz constant of $\mathbf{K}$ (which does not depend on $(x,v)$). By the Kantorovich characterisation of the 1-Wasserstein distance given in \eqref{eq:Kantorovich_distance} we have that
$$\left| \mathbf J(t,x,v, \mu^N_t) - \mathbf J(t,x,v, f_t) \right| \leq C(\mathbf K) W_1(\mu^N_t,f).$$
We take first the supremum over $(x,v)$ on the previous expression, but observe that the right-hand-side is independent of $(x,v)$ and then we take expectations. The result is that
\be \label{eq:estimate1}
\mathbb{E} \left[\sup_{t\in[0,T]}\sup_{(x,v)} \left|\mathbf J(t,x,v,\mu^N_t)-\mathbf J(t,x,v,f_t)\right|\right] \leq C(\mathbf K) \mathbb{E}\left[\sup_{t\in[0,T]} W_1(\mu^N_t, f_t)\right].
\ee
By the triangle inequality, it holds that
$$W_1(\mu^N_t, f_t) \leq W_1(\mu^N_t, \bar \mu^N_t) + W_1(\bar \mu^N_t, f_t),$$
where $\bar \mu^N_t$ denotes the empirical measure associated to the nonlinear process \eqref{eq:auxiliary}. Now, applying Eq. \eqref{eq:distance_empirical} we have that
\beqar 
&&\mathbb{E}\sup_{t\in[0,T]}W_1(\mu^N_t, \bar \mu^N_t) \leq \mathbb{E}\sup_{t\in[0,T]}W_2(\mu^N_t, \bar \mu^N_t)\\
&&\leq \mathbb{E}\sup_{t\in[0,T]}\lp \frac{1}{N}\sum^N_{i=1}\lp|X^{i,N}_t - \bar X^{i,N}_t|^2 + |V^{i,N}_t - \bar V^{i,N}_t|^2 \rp\rp^{1/2}\\
&&\leq \sup_{1\leq i\leq N}\mathbb{E}\sup_{t\in[0,T]} \lp|X^{i,N}_t - \bar X^{i,N}_t|^2 + |V^{i,N}_t - \bar V^{i,N}_t|^2 \rp:=\tilde \eps(N),
\eeqar
where in the first inequality we used the identity \eqref{eq:distance_empirical}  and in the second inequality we used the Cauchy-Schwarz inequality to get rid of the square root and that $\sup (|b| +|c|)\leq \sup |b| + \sup |c|$ and $\sup |b|^2 = (\sup |b|)^2$ for $b,c\in \R$.
We also remind that $W_1\leq W_2$.
It also holds that
$$\mathbb{E}\sup_{t\in[0,T]}W_1(\bar \mu^N_t, f_t) \leq \mathbb{E}\sup_{t\in[0,T]}W_2(\bar \mu^N_t, f_t)  \leq  \mathbb{E}\left[ \sup_{t\in[0,T]} W_2(\bar \mu^N_t, f_t)^2  \right]^{1/2}.$$
To bound this last quantity we will consider the path solutions to the auxiliary particle system  \eqref{eq:auxiliary}, i.e, the path solutions  $(\bar X^{i}_t, \bar V^{i}_t)_{t\in[0,T]}$ as $\mathcal{C}:= C([0,T],~\Omega~\times~\mathbb{S}^{d-1})$-valued random variables. Denote by $f_{[0,T]}\in \mathcal{P}(\mathcal{C})$ their common law and its associated empirical measure
$$\bar\mu^N_{[0,T]}:= \frac{1}{N}\sum_{i=1}^N \delta_{(\bar X^{i}_t, \bar V^{i}_t)_{t\in[0,T]}} \in \mathcal{P}(\mathcal{C}).$$
Since the space $\mathcal{C}$ with the norm $\|w\|_{\infty}:= \sup_{t\in [0,T]} |w_t|$ is a Banach space, one can define the Wasserstein distance on $\mathcal{P}(\mathcal{C})$:
\beqar
&&W^{(T)}_2(m_1,m_2) = \inf \Big\{ \left[ \int \|w_1-w_2\|_\infty^2 \, m(dw_1, dw_2)\right]^{1/2} \\
&&\qquad \quad m\in \mathcal{P}_2(\mathcal{C}\times \mathcal{C})\, \mbox{ with marginals $m_1$ and $m_2$}\Big\}.
\eeqar
One can check that for $m, m'\in \mathcal{P}(\mathcal{C})$ it holds that $W_2(m_t, m'_t) \leq W_2^{(t)}(m,m') \leq W_2^{(T)}(m,m')$ for $t\in[0,T]$.
Therefore 
$$\mathbb{E}\left[ \sup_{t\in[0,T]} W_2(\bar \mu^N_t, f_t)^2  \right] \leq \mathbb{E}\left[ W_2^{(T)}(\bar \mu^N_{[0,T]}, f_{[0,T]})^2  \right] \to 0 \mbox{ as } N\to\infty,$$
where the limit is consequence of  Lem. \ref{lem:law_large_number} (notice that following the proof in \cite{carmona2016lectures} this lemma can be applied to random variables in $\mathcal{C}$). As a consequence we have that
$$\mathbb{E}\sup_{t\in[0,T]}W_1(\bar \mu^N_t, f_t)\to 0 \mbox{ as }N\to \infty.$$
 From these estimates, we have that
 \be \label{eq:estimate2}
 \mathbb{E}[\sup_{t\in[0,T]} W_1(\mu^N_t,f)]\to 0 \quad \mbox{ as }N\to \infty,
 \ee
 since $\tilde \eps(N)\to 0$ as $N\to \infty$ by Th. \ref{th:convergence_process}. \\
 Finally, combining \eqref{eq:estimate0}, \eqref{eq:estimate1} and \eqref{eq:estimate2}, we conclude \eqref{eq:first_inequality}.


\subsection{Proof of Proposition \ref{prop:results-non-singular}}
\label{sec:prop_non_singular}

The proof of Proposition \ref{prop:results-non-singular} follows closely the methodology in \cite{BolCanCar2} and \cite{carmona2016lectures} which are based on Sznitman approach \cite{sznitman1991topics}. The main difference is that the interaction rate $\nu$ and the noise coefficient $\sigma$ are considered to be constant in \cite{BolCanCar2}. In particular, this gives rise to an extra term in the equations (coming from \eqref{eq:extra_term}).

\noindent \textbf{Step 1. Regularised version of the non-singular dynamics}\\
The non-singular particle dynamics \eqref{eq:non-singular-particle_dyn} written in It\^o's convention corresponds to (see Section \ref{sec:Ito_conversion} for more details):
\begin{subequations} \label{eq:particles_Ito}
\begin{numcases}{}
dX^{i,N} = V^{i,N} dt,\\
dV^{i,N} = \nu(\mu^N)\, \mathbf{P}_{(V^{i,N})^\perp} (\tau_0(X^{i,N}, V^{i,N},\mu^N)) dt \nonumber\\
\qquad\qquad+ \sqrt{2\sigma(\mu^N)}\, \mathbf{P}_{(\oin)^\perp}dB^i_t\nonumber\\
\qquad \qquad+(d-1) \sigma(\mu^N) \frac{\oin }{|\oin|^2} dt.
\end{numcases}
\end{subequations}
And the It\^o formulation for the non-singular auxiliary process \eqref{eq:non-singular-auxiliary} is given by
\begin{subequations}\label{eq:auxiliary_Ito}
\begin{numcases}{} 
d\bxin = \boin dt,\\
d\boin = \nu(f)\, \mathbf{P}_{(\boin)^\perp}(\tau_0(\bxin,\boin,f))dt \nonumber\\
\qquad\qquad+ \mathbf{P}_{(\boin)^\perp}[(\nabla_v\sigma(\mu^N))(\bxin,\boin)]\, dt\nonumber\\
\qquad\qquad+ \sqrt{2\sigma(f)}\, \mathbf{P}_{(\boin)^\perp} dB^i_t\nonumber\\
\qquad \qquad+(d-1) \sigma(f)\frac{\boin}{|\boin|^2}dt,\nonumber\\
f_t = \mbox{law}(\bxin, \boin).
\end{numcases}
\end{subequations}

\begin{remark}
Notice that the solutions of \eqref{eq:particles_Ito} and \eqref{eq:auxiliary_Ito} fulfil $|V_t|^2= |V_0|^2$ in the velocities for all times where the solution is defined (this is shown as in Lemma \ref{lem:norm_1_particles}).
\end{remark}

\noindent\textbf{Step 2. Existence and uniqueness for the regularised particle system - Proof of part (i) of Theorem \ref{th:existence_solutions} for \eqref{eq:non-singular-particle_dyn}.}\\
We consider now a regularised version of Systems \eqref{eq:particles_Ito} and \eqref{eq:auxiliary_Ito} using two functions $\tau_1$ and $\tau_2$ both Lipschitz and bounded and satisfying:
\beqar
\tau_1(v) &=& \mathbf{P}_{v^\perp}= \mbox{Id} - \frac{v\otimes v}{|v|^2}, \qquad\mbox{if } |v| \geq 1/2,\\
\tau_2(v) &=& \frac{v}{|v|^2}, \qquad \mbox{if } |v|\geq 1/2. 
\eeqar
With these functions we defined the \textit{regularised particle dynamics} as
\begin{subequations}\label{eq:particle_regular}
\begin{numcases}{} 
dX^{i,N} = V^{i,N} dt,\\
dV^{i,N} = \nu(\mu^N)\, \tau_1(V^{i,N}) (\tau_0)(X^{i,N}, V^{i,N},\mu^N) dt\nonumber\\
\qquad\qquad+ \tau_1(V^{i,N})[(\nabla_v\sigma(\mu^N))(\bxin,\boin)]\, dt \nonumber \\
\qquad\qquad + \sqrt{2\sigma(\mu^N)}\, \tau_1(\oin)dB^i_t\nonumber\\
\qquad \qquad+(d-1) \sigma(\mu^N)\tau_2(\oin) dt.
\end{numcases}
\end{subequations}

\begin{remark} Notice that the functions $\tau_0$, $\tau_1$ and $\tau_2$ are introduced to regularise the original system, in the sense that we obtain a  new system where all the coefficients are Lipschitz and bounded in $\mathcal{H}$. This regularity allows us to apply classical results of existence of solutions and convergence, as we will see next. At the same time, when $|V|=1$ and $|\mathbf{J}|\geq \eps_0$ we recover the approximated particle equations. 
\end{remark}

Firstly,  we have by Lemma \ref{lem:lipschitz_property} that $\tau_0$ is a Lipschitz bounded function and, moreover, all the coefficients are also Lipschitz bounded (using that the product of Lipschitz bounded functions is Lipschitz bounded, see also Remark \ref{rem:lipschitz_condition}). So we have existence and uniqueness of pathwise solutions for \eqref{eq:particle_regular} (see \cite[Theorem 1.2]{carmona2016lectures}, which, for completeness we have added in a simplified form in Theorem \ref{th:existence_particles_carmona} in the Appendix.).

Secondly, one can check as in Lemma \ref{lem:norm_1_particles} that $|\oin_t| =|\oin_0 |=1$. Therefore, the solution to the regularised system \eqref{eq:particle_regular} is also a solution to the non-singular system \eqref{eq:non-singular-particle_dyn}.

\bigskip

\noindent\textbf{Step 3. Existence and uniqueness for an auxiliary regularised  process of \eqref{eq:auxiliary} - Proof of part (ii) of Theorem \ref{th:existence_solutions} for \eqref{eq:non-singular-auxiliary}.}\\
Similarly as before, we consider a regularised version of the non-singular auxiliary process \eqref{eq:non-singular-auxiliary} given by:
\begin{subequations}\label{eq:auxiliary_regularised}
\begin{numcases}{} 
d\bxin = \boin dt,\\
d\boin = \nu(f)\, \tau_1(\boin)(\tau_0(\bxin,\boin,f)dt \nonumber \\
\qquad\quad+ \tau_1(\bar V^{i,N})[(\nabla_v\sigma(\mu^N))(\bxin,\boin)]\, dt\\
\qquad\qquad+ \sqrt{2\sigma(f)}\,\tau_1(\boin) dB^i_t \nonumber\\
\qquad\qquad +(d-1) \sigma(f) \tau_2(\boin)dt \\
f_t = \mbox{law}(\bxin, \boin).
\end{numcases}
\end{subequations}

Since \eqref{eq:auxiliary_regularised} has bounded and Lipschitz coefficients in $\mathcal{H}$  (this can be seen as in the previous proof), it admits a pathwise unique global solution (this statement is shown in \cite[Theorem 1.7]{carmona2016lectures}, which is a generalisation of Sznitmann's strategy \cite[Theorem 1.1]{sznitman1991topics} - see Theorem \ref{th:existence_auxiliary_problem_nonlinear} in the Appendix). We can prove as in Lemma \ref{lem:norm_1_particles} that $d|\bar V_t|^2=0$ so $|\bar V_t|=1$ for all times. Therefore, the solution to \eqref{eq:auxiliary_regularised} is also solution of the non-singular auxiliary system \eqref{eq:non-singular-auxiliary}.

\bigskip

\noindent\textbf{Step 4. Existence and uniqueness for the non-singular kinetic equation \eqref{eq:non-singular-kinetic} - Proof of part $(iii)$ of Theorem \ref{th:existence_solutions} for the non-singular case: \eqref{eq:non-singular-auxiliary} and  \eqref{eq:non-singular-kinetic}}\\
Next we show that part $(iii)$ of Theorem \ref{th:existence_solutions} holds for the non-singular auxiliary process \eqref{eq:non-singular-auxiliary} and the non-singular kinetic equation \eqref{eq:non-singular-kinetic}.

We can apply Theorem \ref{theo:local cauchy} to the non-singular kinetic equation \eqref{eq:non-singular-kinetic}. Moreover, since $\nu$ is bounded and $\tau_0$ is also bounded
we have global-in-time existence of solutions for \eqref{eq:non-singular-kinetic}.

Now, we also know that the kinetic equation \eqref{eq:general kinetic orthogonal} has existence and uniqueness of  weak solutions in $L^\infty(~[0,~T_{max}),~L^1_{x,\omega}~\cap~L^p_{x,\omega})$ for $p$ in $[2,\infty]$. Let $f_t$ be the weak solution of the kinetic equation for $t\in [0, T_{max})$. We also now that for $\eps_0>0$ small enough there is a time $T_{\eps_0}\leq T_{max}$ such that $|\mathbf{J}_{f_{t}}|>\eps_0$ for all $t\in[0, T_{\eps_0})$. Therefore, for $t\in[0, T_{\eps_0})$ the solution $f_t$ of the kinetic equation is also solution of the non-singular kinetic equation \eqref{eq:non-singular-kinetic}. 

The fact that the Fokker-Planck equation for the non-singular auxiliary process $(\bar X_t, \bar V_t)$ \eqref{eq:non-singular-kinetic} corresponds to the non-singular kinetic equation \eqref{eq:non-singular-kinetic} is proven in \cite{BolCanCar2} (this is an application of It\^o's formula, see for example  \cite[Remark 1.2]{sznitman1991topics}, \cite[Remark 1.8]{carmona2016lectures} and Eq. \eqref{eq:ito_application}). The proof that $(\bar X_t, \bar V_t)$ has law $f_t$ (local-in-time) is carried out in the same way as in \cite{BolCanCar2}.

\bigskip
\noindent \textbf{Step 5. Mean-field limit  for the non-singular dynamics - Proof of Theorem \ref{th:convergence_process} for non-singular dynamics \eqref{eq:non-singular-particle_dyn} and \eqref{eq:non-singular-auxiliary}.}\\
The proof of Theorem \ref{th:convergence_process} is an application of Sznitman approach, generalised in \cite{carmona2016lectures}. 

We can apply the result in \cite[Theorem 1.10]{carmona2016lectures} (a condensed version can be found in Theorem \ref{th:particle_approximation_carmona} in the Appendix).

\begin{remark}
The control in the particle position is immediate from the one in the velocities:
$$|\xin_T -\bxin_T|^2 = \left|\int^T_0 \oin_s \, ds -\int^T_0 \boin_s \right|^2 \leq T \sup_{0\leq t\leq T} |\oin_t- \boin_t|^2.$$

\end{remark}

\appendix

\section{Some known results on SDE}
\label{sec:carmona}

For the sake of completeness, we introduce in this section known results for stochastic differential equations that we used in the main part of the document.

\subsection{Results on existence of solutions and large-particle limits}

The results and proofs from this section can be found in a more general form in \cite{carmona2016lectures}. The original approach to proving the large-particle limit is due to Sznitmann and can be found in \cite{sznitman1991topics}.

\textit{General stochastic differential equations.}
Consider a filtered probability space $(\Omega, \mathcal{F}, (\mathcal{F}_t)_{t\geq 0}, \mathbb{P})$ and an $(\mathcal{F}_t)_{t\geq 0}$-Brownian motion $B$ taking values in $\R^m$.
 Consider the stochastic differential equation giving the evolution for $Z_t\in \R^{p}$, for $t\geq 0$:
\begin{equation}
\label{eq:SDE_model}
dZ_t = b(Z_t) dt + a(Z_t) dB_t,
\end{equation}
with initial data $Z(t=0)=Z_{0}\in \R^p$,
where the coefficients $a:\R^p \to \R^{p\times m}$ and $b:\R^p\to \R^p$ are Lipschitz and bounded. (This is the form taken by the SDE \eqref{eq:particle_regular} with $p=2d$ and $Z_t=(X_t,V_t)$.)

\begin{definition}[See Definition 1.1 in \cite{carmona2016lectures}] 
An $(\mathcal{F}_t)_{t\geq 0}$- adapted continuous process $(Z_t)_{0\leq t\leq T}$ is a solution to \eqref{eq:SDE_model} if
$$Z_t = Z_0 +\int^T_0 b(Z_s)\, ds +\int^T_0 a(Z_s)\, dB_s, \quad 0\leq t\leq T.$$

\end{definition}

\begin{theorem}[Adapted from Theorem 1.2   \cite{carmona2016lectures} ] 
\label{th:existence_particles_carmona}
Let us assume that $Z_0 \in L^2$ is independent of $B$ and that the coefficients $b$ and $a$ are Lipschitz and bounded. Then, there exists a  unique solution of \eqref{eq:SDE_model}. Moreover, $Z_t\in L^2$ for all $t<T$, with $T$ finite.
\end{theorem}

\noindent\textit{Results for non-linear SDE.}\\
We consider a nonlinear SDE of the form
\begin{equation}
\label{eq:auxiliary_general}
dZ_t = b(Z_t, \mathcal{L}(Z_t)) dt + a(Z_t, \mathcal{L}(Z_t) ) dB_t.
\end{equation}
where $\mathcal{L}(Z)$ denotes the distribution or  law of the random element $Z$. Here we also assume that $b: \R^p \times \mathcal{P}_2(\R^p) \to \R^p$ and $a: \R^p \times \mathcal{P}_2(\R^p) \to \R^{p\times m}$ are bounded and Lipschitz. Specifically, in this case to be Lipschitz means that for all $z,z'\in \R^d$ and all $\mu, \mu'\in \mathcal{P}(\R^d)$ it holds that
$$|(b,a)(z, \mu) - (b,a)(z',\mu')|\leq c(|z-z'| + W^{(2)}(\mu, \mu')),$$
for some constant $c$ (see Section \ref{sec:preliminaries_SDE} for more details).
System \eqref{eq:auxiliary_general} is the form taken by the system auxiliary system \eqref{eq:auxiliary_regularised}.

\begin{theorem}[From Theorem 1.7 in \cite{carmona2016lectures}]
\label{th:existence_auxiliary_problem_nonlinear}
Let us assume that $Z_0\in L^2$ is independent of $B$, and that the coefficients $b$ and $a$ are bounded and Lipschitz. Then, there exists a unique solution to \eqref{eq:auxiliary_general}.
\end{theorem}

\bigskip

\noindent For $\varphi\in C^2_b(\R^p)$, It\^o's formula applied to the process $Z_t$ gives  (see Remark 1.8 in \cite{carmona2016lectures}):
\beqarl
\varphi(Z_t) &=& \varphi(Z_0)+\int^t_0\Big[ \frac{1}{2}\mbox{trace}[a(Z_s, \mathcal{L}(Z_s))^T a(Z_s, \mathcal{L}(Z_s))] D^2\varphi(Z_s)\nonumber\\
&& \qquad+b(Z_s, \mathcal{L}(Z_s))D\varphi(Z_s)\Big] ds + \int^t_0 D\varphi(Z_s) \, a(Z_s, \mathcal{L}(Z_s)) \, dB_s, \label{eq:ito_application}
\eeqarl
where the exponent $T$ denotes the transpose and $D^2$ denotes the Hessian matrix.

\begin{lemma}[Law of large numbers: Lemma 1.9 in \cite{carmona2016lectures}]
\label{lem:law_large_number}
Let $\mu\in \mathcal{P}_2(\R^p)$, and let $(Z_i)_{i\in \mathbb{N}}$ be a sequence of independent identically distributed random variables with common law $\mu$. For each $N\geq 1$ denote by $\mu^N$ the empirical distribution associated to the first $N$ elements of the sequence, i.e.
$$\mu^N(z)= \frac{1}{N}\sum_{i=1}^N \delta_{Z_i}(z).$$
 Then, it holds that
 \be \label{eq:limitW2empirical}
\lim_{N\to\infty}\mathbb{E}[W_2(\mu^N, \mu)^2] =0.
\ee
\end{lemma}

\bigskip

Let $z_i \in \R^p$ and $z_i' \in \R^p$ for $i=1,\hdots,N$, it holds (\cite[(1.24)]{carmona2016lectures}) 
\be \label{eq:distance_empirical}
W^{(2)}\lp \frac{1}{N}\sum^N_{i=1}\delta_{z_i},\, \frac{1}{N}\sum^N_{i=1}\delta_{z_i'} \rp \leq \lp\frac{1}{N}\sum^N_{i=1}|z_i-z_i'|^2 \rp^{1/2}.
\ee

\bigskip
\noindent \textit{Particle approximations (from \cite[Section 1.3.4]{carmona2016lectures}).}

\begin{theorem}[Extracted and adapted from Section 1.3.4 in \cite{carmona2016lectures}]
\label{th:particle_approximation_carmona}
Consider
\be  \label{eq:integral_auxiliary}
Z^i_t = Z^i_0 + \int^T_0 b(Z^i_s, \mathcal{L}(Z^i_s)), ds +\int^T_0 a(Z^i_s, \mathcal
L (Z^i_s))\, dB^i_s,
\ee
and
\be \label{eq:integral_particle}
Z^{i,N}_t = Z^i_0 +\int^T_0 b(Z^{i,N}_s, \mu^N_s)\, ds + \int^T_0 a(Z^{i,N}_s, \mu^N_s)\, dB^i_s,
\ee
where $\mu^N_s$ is the empirical distribution of the $N$ particles. Assume that $a, b$ are bounded and Lipschitz in $\R^p\times \mathcal{P}_2(\R^p)$. Then, it holds that
\be \label{eq:limit_particles_carmona}
\lim_{N\to\infty}\sup_{1\leq i\leq N} \mathbb{E}\left[ \sup_{0\leq t\leq T} |Z^{i,N}_t - Z^{i}_t|^2\right] =0.
\ee
\end{theorem}

\subsection{Stratonovich to Ito's convention}
\label{sec:Ito_conversion}

We use the results presented in \cite{evans2012introduction} as well as \cite[V.30, Theorem (30.14)]{rogers2000diffusions}. Particularly, if $Z_t$ is solution to the Stratonovich SDE
\begin{subequations} \label{eq:SDE_stratonovich}
\begin{numcases}{}
dZ = b(Z,t)dt + a(Z,t) \circ dB_t,\\
Z(t=0)= Z_0,
\end{numcases}
\end{subequations}
where $B$ is an $m$- dimensional Brownian motion, and $b: \R^p\times [0,T] \to \R^p$ and $a:\R^p\times[0,T]\to M^{p\times m}$ (the space of real matrices of dimension $p\times m$) and are such that there is existence and uniqueness of solutions for the SDE \eqref{eq:SDE_stratonovich}.
Then, $Z_t$ is solution to the It\^o SDE:
\begin{subequations}
\begin{numcases}{}
dZ = [b(Z,t)+\frac{1}{2}c(Z,t)]\, dt + a(Z, t) dB_t,\\
Z(t=0)= Z_0,
\end{numcases}
\end{subequations}
where
$$c_i(z,t) = \sum_{k=1}^m\sum_{j=1}^n \frac{\partial a_{ik}}{\partial z_j}(z,t) a_{jk}(z,t).$$
In our case $a: \R^{2d}\to \R^{2d\times d}$
with
$$a(x,v) =
\lp
\begin{array}{c}
0_{d\times d} \\
\eta(x,v)
\end{array}
\rp,
$$
where $0_{d\times d}$ is a $d\times d$ zero-matrix and
$$\eta(x,v) = \alpha(x,v) P_{v^\perp}; \qquad \alpha(x,v) = \sqrt{2\sigma(f)(x,v)}.$$
With this we have that $c:\R^{2d}\to \R^{2d}$ with
$c_i = 0$ for $i=1,\hdots, d$ since $\gamma_{jk}=0$ for all $j=1,\hdots, d$ and any $k=1,\hdots, d$. Now, for $i=d+1,\hdots, 2d$ we have that
$$c_i(x,v) = \sum_{k=1}^d \sum_{j=1}^d \frac{\partial \eta_{ik}}{\partial v_j}(x,v) \eta_{jk}(x,v).$$
Using that 
$$\eta_{jk}(x,v) = \alpha(x,v) \lp \delta_{j=k} - \frac{v_j v_k}{|v|^2}\rp,$$
we compute the previous expression and obtain that
$$c(x,v) = \lp 0_{1\times d},\,  2 (d-1) \sigma(f) (x,v) \frac{v}{|v|^2} + P_{v^\perp}[(\nabla_v\sigma(f))(x,v)] \rp.$$

\bibliographystyle{acm}
\bibliography{bibliography_collective_behaviour}
\bigskip
\signmb
\signad
\signsm

\end{document}